\documentclass[11pt, reqno]{amsart}
\usepackage{amssymb,amsmath,epsfig,mathrsfs, enumerate, xparse, mathtools}
\usepackage[french, english]{babel}
\usepackage{lipsum}
\usepackage[pagebackref,colorlinks=true,linkcolor=blue,citecolor=blue]{hyperref}

\usepackage{amsmath}
\usepackage{amsfonts}
\usepackage{amssymb}
\usepackage{amsthm}
\usepackage{epsfig,graphics,mathrsfs}
\usepackage{graphicx}
\usepackage{dsfont}

\usepackage[usenames, dvipsnames]{color}

\usepackage[pagewise]{lineno}
\usepackage[nodisplayskipstretch]{setspace}
\usepackage{graphicx}\usepackage[normalem]{ulem}
\usepackage{fancyhdr}
\usepackage[margin=2.5cm]{geometry}
\usepackage{hyperref}
\hypersetup{
	colorlinks   = true,
	urlcolor     = blue,
	linkcolor    = blue,
	citecolor   = red ,
	bookmarksopen=true
}

\theoremstyle{plain}
\newtheorem{thrm}{Theorem}[section]
\newtheorem{lemma}[thrm]{Lemma}
\newtheorem{prop}[thrm]{Proposition}
\newtheorem{cor}[thrm]{Corollary}

\usepackage[usenames]{color}

\allowdisplaybreaks
\begin{document}
	\newcommand{\sn}{\mathbb{S}^{n-1}}
	\newcommand{\SL}{\mathcal L^{1,p}( D)}
	\newcommand{\Lp}{L^p( Dega)}
	\newcommand{\CO}{C^\infty_0( \Omega)}
	\newcommand{\Rn}{\mathbb R^n}
	\newcommand{\Rm}{\mathbb R^m}
	\newcommand{\R}{\mathbb R}
	\newcommand{\Om}{\Omega}
	\newcommand{\Hn}{\mathbb H^n}
	\newcommand{\A}{\alpha }
	\newcommand{\B}{\beta}
	\newcommand{\eps}{\ve}
	\newcommand{\BVX}{BV_X(\Omega)}
	\newcommand{\p}{\partial}
	\newcommand{\IO}{\int_\Omega}
	\newcommand{\bG}{\boldsymbol{G}}
	\newcommand{\bg}{\mathfrak g}
	\newcommand{\bz}{\mathfrak z}
	\newcommand{\bv}{\mathfrak v}
	\newcommand{\Bux}{\mbox{Box}}
	\newcommand{\e}{\ve}
	\newcommand{\X}{\mathcal X}
	\newcommand{\Y}{\mathcal Y}
	\newcommand{\Z}{\mathcal Z}
	\newcommand{\I}{\mathcal I}
	\newcommand{\la}{\lambda}
	\newcommand{\vf}{\varphi}
	\newcommand{\rhh}{|\nabla_H \rho|}
	\newcommand{\Ba}{\mathcal{B}_\beta}
	\newcommand{\Za}{Z_\beta}
	\newcommand{\ra}{\rho_\beta}
	\newcommand{\n}{\nabla}
	\newcommand{\vt}{\vartheta}
	\newcommand{\its}{\int_{\{y=0\}}}
	\newcommand{\py}{\partial_y^a}
\newcommand{\sa}{\langle}
\newcommand{\da}{\rangle}
\newcommand{\mi}{\mathscr I}
\newcommand{\F}{\mathscr F}
	
	\numberwithin{equation}{section}

	\newcommand{\RN} {\mathbb{R}^N}
	\newcommand{\Sob}{S^{1,p}(\Omega)}
	\newcommand{\Dxk}{\frac{\partial}{\partial x_k}}
	\newcommand{\Co}{C^\infty_0(\Omega)}
	\newcommand{\Je}{J_\ve}
	\newcommand{\beq}{\begin{equation}}
		\newcommand{\bea}[1]{\begin{array}{#1} }
			\newcommand{\eeq}{ \end{equation}}
		\newcommand{\ea}{ \end{array}}
	\newcommand{\eh}{\ve h}
	\newcommand{\Dxi}{\frac{\partial}{\partial x_{i}}}
	\newcommand{\Dyi}{\frac{\partial}{\partial y_{i}}}
	\newcommand{\Dt}{\frac{\partial}{\partial t}}
	\newcommand{\ds}{\displaystyle}
	\newcommand{\Zt}{{\mathcal Z}^{t}}
	\newcommand{\ve}{\varepsilon}
	\newcommand{\D}{\operatorname{div}}
	\newcommand{\G}{\mathscr{G}}
	\newcommand{\w}{\tilde{w}}
	\newcommand{\s}{\sigma}

\title[Sharp order of vanishing, etc.]{Sharp order of vanishing for parabolic equations, nodal set estimates and   Landis type results}

\dedicatory{This work is dedicated to the memory of our dear friend and colleague Luis Escauriaza}	
	
\author{Vedansh Arya}
\address{Department of Mathematics and Statistics, University of Jyväskylä, Finland}\email[Vedansh Arya]{vedansh.v.arya@jyu.fi}
	
\author{Agnid Banerjee}
\address{School of Mathematical and Statistical Sciences\\ Arizona State University \\ Tempe, AZ 85287-1804\\USA}\email[Agnid Banerjee]{agnid.banerjee@asu.edu}

\author{Nicola Garofalo}
\address{Department of Civil, Environmental and Architectural Engineering\\University of Padova \\ Italy}\email[Nicola Garofalo]{nicola.garofalo@unipd.it}



\medskip

\thanks{N. Garofalo is supported in part by a Progetto SID (Investimento Strategico di Dipartimento): ``Aspects of nonlocal operators via fine properties of heat kernels", University of Padova (2022); and by a PRIN (Progetto di Ricerca di Rilevante Interesse Nazionale) (2022): ``Variational and analytical aspects of geometric PDEs". He is also partially supported by a Visiting Professorship at the Arizona State University}

\maketitle
	
\selectlanguage{english}

\begin{abstract}
We establish a new sharp estimate of the order of vanishing of solutions to parabolic equations with variable coefficients. For real-analytic leading coefficients, we prove a localised estimate of the nodal set, at a given time-level, that generalises the celebrated one of Donnelly and Fefferman. We also establish Landis type results for global solutions.
\end{abstract}
	
\tableofcontents
	
\section{Introduction}
	
The subject of unique continuation for second-order partial differential equations occupies a central position in analysis and geometry. More in  particular, quantitative uniqueness for elliptic equations has undergone a great development during the past three decades, especially in connection with the following conjecture of Yau in \cite{Yau}: let $M$ be a $n$-dimensional, $C^\infty$ Riemannian manifold compact and without boundary, and consider an eigenfunction $\Delta \vf = - \la \vf$ on $M$. Then the $(n-1)-$dimensional measure of the nodal set of $\vf$ satisfies the estimate: 
\begin{equation}\label{yau}
c\ \sqrt \la \le H^{n-1}(\{x\in M\mid \vf(x) = 0\}) \le C \sqrt \la,
\end{equation}
 where $C, c>0$ are two constants depending only on $n$, the Riemannian metric, but not the eigenvalue $\la$. This order of vanishing is sharp since, given the harmonic function $P_\kappa(x)= \Re (x_1 + i x_2)^\kappa$ in $\R^{n+1}$, then $\vf_\kappa = {P_\kappa}\big|_{\mathbb S^{n}}$ satisfies the well-known equation $\Delta_{\mathbb S^{n}} \vf_\kappa = - \lambda_\kappa \vf_\kappa$, with $\lambda_\kappa = \kappa(\kappa+n-1)$, and the order of vanishing of $\vf_\kappa$ at the North pole $(0,...,0,1)$ is precisely $\kappa \approx \sqrt {\lambda_\kappa}$. 
In their famous papers \cite{DF1}, \cite{DF2}, Donnelly and Fefferman showed that if $u$ solves $\Delta u = - \la u$ in a $C^\infty$, compact and connected $n$-dimensional Riemannian manifold $M$, then  the maximal vanishing order of $u$ is $\approx \sqrt{\lambda}$. Using their estimate, the same authors proved that, if the Riemannian metric is real-analytic, then \eqref{yau} does hold, thus giving a complete  answer to Yau's conjecture in the analytic category. More recently, thanks to the breakthrough  works of Logunov and Malinnikova in \cite{L1}, \cite{L2} and \cite{LM}, there have been some important developments in the $C^\infty$ setting as well. The reader should also see \cite{CLMM} for a sharp local Courant nodal domain estimate. For a beautiful account on quantitative unique continuation, Yau's conjecture and recent developments sparked by it, we refer the reader to \cite{LMqucp} and \cite{LMyau}.

In this paper, we establish a \emph{space-like} quantitative uniqueness result for nontrivial solutions to parabolic equations in $\Rn\times [0,\infty)$ in the form\footnote{We alert the reader that, in order to simplify the presentation, we have chosen to work with a backward parabolic equation in a forward time-slab. Consequently, all results - including the existing ones from other authors -  will be stated in accordance with this choice. Using the change of variable $t\to - t$, the reader can readily convert all the statements into ones for a forward parabolic equation 
\[
u_t - \D(A(x,t)\n u) +\langle b(x,t), \nabla u\rangle + V(x,t)u=0,
\]
in a backward time-slab $\Rn\times (-\infty,0]$.}
\begin{align}\label{meq}
Lu = u_t+\D(A(x,t)\n u) +\langle b(x,t), \nabla u\rangle + V(x,t)u=0,
\end{align}
under minimal regularity hypothesis on the coefficients $A(x,t) = [a_{ij}(x,t)]$, $b(x,t) = (b_1(x,t),...,b_n(x,t)$ and $V(x,t)$.
The case  $b(x,t) \equiv 0$ in \eqref{meq} was recently settled by two of us in \cite{AB}, but their study left open the treatment of the drift. One of the main new contributions of the present work is to fill this gap. This requires a delicate analysis which complements and completes the available theory. The relevant assumptions on the symmetric, matrix-valued function  $(x,t)\to A(x,t) \in M_{n\times n}(\R)$, the drift  $b(x,t)$ and the potential $V(x,t)$ are as follows:
\begin{itemize}
\item[(i)] There exist $\Lambda \geq 1$  such that for all $x, y\in \Rn$, and $s, t\in [0,\infty)$, one has for every $\xi\in \Rn$
\begin{equation}\label{ell}
\Lambda^{-1} |\xi|^2 \leq \sa A(x,t)\xi,\xi\da\leq \Lambda |\xi|^2.
\end{equation}
\item[(ii)] There exist $M\geq 0$ such that for all $x, y\in \Rn$, and $s, t\in [0,\infty)$, one has
\begin{equation}\label{ass}
|a_{ij}(x,t) - a_{ij}(y,s) |\leq M(|x-y| + |t-s|^{1/2}).
\end{equation}
\end{itemize}
In what follows we indicate by $B_r(x)$ the ball of radius $r$ with centre at $x \in \R^n$, $Q_r(x,t)$ will denote the space-time  cylinder $B_r(x) \times [t,t+r^2].$ When $(x,t) = (0,0)$, we simply write $B_r$ and $Q_r$. With $b(x,t)$ and $V(x,t)$ as in \eqref{ass} above, we define 
$$||V||_{1,1/2} = ||V||_{L^{\infty}(Q_4)} +||\n_x V||_{L^{\infty}(Q_4)} +\underset{(x,t), (x,s) \in Q_4}{\sup} \frac{|V(x,t)-V(x,s)|}{|t-s|^{1/2}}$$ and 
$$||b||_{1,1/2} = \sup_{i=1,...,n}\left(||b_i||_{L^{\infty}(Q_4)} +||\n_x b_i||_{L^{\infty}(Q_4)} + \underset{(x,t), (x,s) \in Q_4}{\sup} \frac{|b_i(x,t)-b_i(x,s)|}{|t-s|^{1/2}}\right).
$$
Throughout this paper, we assume that $||b||_{1,1/2}<\infty$, and $||V||_{1,1/2}<\infty$. With these quantities in place, we introduce a normalisation quantity which plays a pervasive role in what follows:
\begin{equation}\label{T}
\Theta =	\frac{\int_{Q_{4}} u^2( x, t)dxdt}{\int_{B_{1}}u^2(x,0)dx}.
\end{equation}  		
In the results in this paper it will always be assumed that the relevant solution does not vanish identically at time $t = 0$. Consequently, the quantity \eqref{T} will be well-defined. The following is our main result on quantitative unique continuation.	

\begin{thrm}[Space-like vanishing order]\label{main}
Let $u$ be a solution to \eqref{meq} in $Q_4$ such that $u(\cdot, 0) \not \equiv  0$ in $B_1$. Then there exists a universal constant $N>0$ such that, with
\begin{equation}\label{K}
\mathcal{K}=\frac{1}{\int_{B_1} u^2(x,0)dx}+ N\operatorname{log}(N\Theta)+N(||V||_{1,1/2}^{1/2}+||b||_{1,1/2}+1),
\end{equation} 
one has for all $r \le 1/2$,  
\begin{align}\label{df}
\int_{B_r} u^2(x,0)dx \geq   r^{\mathcal{K}}.
\end{align}
\end{thrm}

As we have mentioned above, when $b(x,t)\equiv 0$, Theorem \ref{main} was proved in \cite{AB}. We emphasise that the linear dependence on the norm of the drift in the constant $\mathcal K$ in \eqref{K} is sharp, as shown by the following example. Consider $P_\kappa(x)= \Re (x_1 + i x_2)^\kappa$ in $\R^{n}$,  $\vf_\kappa = {P_\kappa}\big|_{\mathbb S^{n-1}}$, and for $(x,y) \in \mathbb S^{n-1} \times \R$, define with $\la = \kappa(\kappa+n-2)$, 
\[
v(x, y)= e^{\sqrt{\frac{\lambda}{2}} y}\ \vf_\kappa(x).
\]
A computation gives
\begin{equation}\label{example}
\Delta_{\mathbb S^{n-1}} v + v_{yy} + \sqrt{\frac{\lambda}{2}}\ v_y = 0.
\end{equation} 
Consider the point $p_0\overset{def}=(n_0, 0) \in \mathbb S^{n-1} \times \R$ where $n_0 \in \mathbb S^{n-1}$ is the north pole. If we denote by $b = \sqrt{\frac{\lambda}{2}}$ the drift in \eqref{example}, then it is readily seen that
the vanishing order of $v$ at $p_0$ is $\kappa \approx \sqrt{\lambda}\approx \sqrt{\frac{\lambda}{2}} = ||b||_{1,1/2}$. 

\medskip

In our next result, Theorem \ref{maindf}, we use Theorem \ref{main} to obtain an estimate which should be regarded as a parabolic counterpart of the above mentioned results in \cite{DF1}, \cite{DF2}.  

\begin{thrm}[Space-like estimate of the nodal set]\label{maindf}
In addition to \eqref{ell}, assume that $A\in C^\infty(\Rn\times \R)$ and that $x\to A(x, t)$ is real-analytic, uniformly in $t$. With $\la>0$, let $u$ be a solution in $Q_4$ to the equation 
\begin{equation}\label{real1}
u_t  + \D(A(x,t) \nabla u) +  \la u =0, 
\end{equation}
such that  $u(\cdot, 0) \neq 0$. Then there exist $C=C(n, A)>0$, and a universal constant $N>0$, such that
\begin{equation}\label{df1}
H^{n-1}(\{x: u(x, 0) =0\} \cap B_{1/2}) \leq C\left(\sqrt \la +N \log(N \Theta) + \frac{1}{\int_{B_1} u(x,0)^2 dx} +1\right),
\end{equation}
where $H^{n-1}$ denotes the $(n-1)$-dimensional Hausdorff measure in $\Rn$.
\end{thrm}
To the best of our knowledge, in the available literature the result most closely related to Theorem \ref{maindf} is the one due to F. H. Lin in \cite[Theorem 4.2]{Lin}. There he proved that, if $(M, g)$ is an analytic compact manifold without boundary, and if $u$ is a non-zero solution in $M \times [0, \infty)$ to the equation
\[
\partial_t u - \Delta_M u =0,
\]
then 
\begin{equation}\label{li1}
H^{n-1}\{x\in M: u(x,t) =0\} \leq C(n, g, M)\ N(t), 
\end{equation} 
where $N(t)$ is an explicitly computed frequency function depending on all the eigenvalues of $-\Delta_M$ (as well as on the initial values $u(x,0)$), see (4.7) in \cite{Lin}. It ensues that \eqref{li1} has a global character, whereas \eqref{df1} is local in nature. We also refer to \cite{HLi} for  results related to \eqref{li1}, but with lower regularity assumptions on the coefficients. 

\medskip

As a further  consequence of Theorem \ref{main}, in the last part of our work we establish some new elliptic-type global results of Landis type for solutions to \eqref{meq}. More precisely, we are interested in deriving lower bounds on the decay rate of global solutions in $\Rn \times [0, \infty)$ to the parabolic equation 
\begin{equation}\label{e0}
u_t + \Delta u  = \sa b(x,t),\nabla u\da + V(x,t) u,
\end{equation}
under different assumptions on $b(x,t)$ and $V(x,t)$. Our results are in part motivated by the work \cite{BGspringer} of two of us, on Modica type gradient bounds for the reaction-diffusion equation in $\Rn\times [0,\infty)$
\begin{equation}\label{rd}
u_t + \Delta u = F'(u).
\end{equation}
In Theorem 4.1 in that paper it was shown that, surprisingly, any bounded solution to \eqref{rd} satisfies the following global gradient bound
\[
|\nabla u(x,t)|^2\le 2 F(u(x,t)),
\]
without any assumption on the right-hand side, other than $F\ge 0$. When $F(u) = \frac{1}4(1-u^2)^2$ is the double-well potential, for which $F'(u) = u^3 - u$, there exist eternal traveling wave solutions $|u(x,t)|\le 1$ of \eqref{rd} of the form
\begin{equation}\label{twave}
u(x',x_n,t) = v(x',x_n - ct),\ \ \ \ \ \ c\ge 0,
\end{equation}
for which $\partial_{x_n} u (x)  > 0$. This led to formulate the following parabolic version of a famous conjecture of De Giorgi which is presently fully open.

\medskip
 
\noindent \textbf{Conjecture:} \emph{Let $u$  be a solution to \eqref{rd} in $\Rn \times [0,\infty)$ 
such that $|u|\le 1$, and $\partial_{x_n} u (x,t)  > 0$ for all $(x,t)  \in \Rn \times [0,\infty)$. Then, $u$ must  be an eternal traveling wave, i.e. after  a  change  of coordinates, $u$ must be of the  form \eqref{twave}}. 

\medskip

Before stating our next results, Theorems \ref{main2}, \ref{main1} and Corollary \ref{landis1}, we mention that, in the time-independent situation, the relevant problems were originally posed by Landis in the late 60's, see \cite{KLa}. He conjectured that if a solution in an exterior domain $\Om = \Rn\setminus K$ to the equation
\begin{equation}\label{f0} 
\Delta u = V(x) u
\end{equation}
decays faster than $e^{-\kappa |x|}$, for some $\kappa > ||V||_{L^{\infty}(\Om)}$, then  $u \equiv 0$ in $\Om$. This should be seen as an unique continuation property at infinity for the Schr\"odinger equation \eqref{f0}. Landis' conjecture was disproved by Meshkov, who constructed a bounded complex-valued $V$ and a nontrivial solution $u$ to \eqref{f0} satisfying $|u(x)| \leq Ce^{-C|x|^{4/3}}$, see \cite{Me}. Bourgain and Kenig  showed in \cite{BK} that if $u$ is a bounded solution to \eqref{f0} with $||V||_{L^{\infty}} \leq 1$, then one has 
\begin{equation}\label{bk}
\int_{B_1(x_0)} u^2(x) dx \geq C e^{-|x_0|^{4/3} \log |x_0|},
\end{equation}
and they used such lower bound in their resolution of Anderson localisation for the Bernoulli problem. We note that \eqref{bk} constitutes a sharp quantitative  decay in view of Meshkov's example.   Although Landis' conjecture presently remains a challenging open problem for real-valued $V$ and $u$, some interesting partial progress exist. Under the assumption $V\geq 0$,  it was proved in \cite{KSW} in the planar case. The sign restriction $V\ge 0$ has been recently removed in \cite{LMNN}, thus the conjecture has been fully resolved in $\R^2$.  For various other partial results on Landis' conjecture and its variants, we refer to  \cite{D}, \cite{DKW}, \cite{DW}, \cite{KW},  \cite{LUW}, \cite{LW},   \cite{Ro}.  We also mention \cite{EKPV}, where it was shown that exponential decay of super-Gaussian type is not possible at infinity for forward parabolic problems.

Returning to global decay estimates for solutions of \eqref{e0}, we first state a result for constant drift and potential.
	
\begin{thrm}\label{main2}
Let $c_0 \in \mathbb C, b_0 \in \mathbb C^n$, and assume that $u$ be a solution in $\Rn \times [0, \infty)$ to the equation
\begin{equation}\label{lan1}
u_t + \Delta u  = c_0 u + \sa b_0, \nabla u\da, 
\end{equation}
such that $u(\cdot, 0) \not \equiv 0$ and that for some $C\ge 0$ one has for $x\in \Rn$ and $t\in [0,\infty)$
\begin{equation}\label{expb1}
|u(x,t)| \leq Ce^{C(|x| + |t|^{1/2})}.
\end{equation}	
Then the following decay estimate holds for $x_0 \in \Rn $ with $|x_0| =R>>1$
\begin{equation}\label{dec1}
\int_{	B_1(x_0)} u^2(x, 0) dx \geq e^{-\tilde C R \log R}.
\end{equation}
where $\tilde C>0$ depends only on $n, C, |b_0|, |c_0|$ and $\int_{B_1} u^2(x, 0)dx$.
\end{thrm}

In the next result, under the hypothesis that the potential $V(x,t)$ be bounded, we obtain a parabolic analogue of the Bourgain-Kenig type estimate in \eqref{bk} above.  	
	
\begin{thrm}[Space-like Bourgain-Kenig type decay]\label{main1}
Suppose that $V\in L^\infty(\Rn\times [0,\infty),\mathbb C)$,  and
let $u$ be a solution in  $\Rn \times [0, \infty)$ to the equation 
\begin{equation}\label{v1}
u_t  + \Delta  u  = V(x,t) u,
\end{equation}
such that $u(\cdot, 0) \not \equiv 0$, and satisfying \eqref{expb1}. Then there exists $C>0$, depending only on $n, M, ||V||_{\infty}$ and $\int_{B_1} u^2(x, 0)dx$, such that the following holds for $x_0 \in \Rn $ with $|x_0| =R>>1$
\begin{equation}\label{dec1}
\int_{	B_1(x_0)} u^2(x, 0) dx \geq e^{-C R^{4/3} \log R}.
\end{equation}
\end{thrm}
	
Again, we note that in view of Meshkov's counterexample,  the estimate  \eqref{dec1} is sharp.
As a corollary of Theorem \ref{main1}, we have the following result.
	
\begin{cor}[Space-like unique continuation at infinity]\label{landis1}
Suppose that $u$ and $V$ satisfy the hypothesis in Theorem \ref{main1}. 
If for some $\ve>0$, 
\begin{equation}\label{dec2} 
|u(x, 0)| \leq Ce^{-|x|^{4/3+ \ve}},
\end{equation} 
then $u \equiv 0$ in $\Rn \times [0, \infty)$.
\end{cor}

\medskip

In closing, we mention that there is a large literature on quantitative unique continuation. While we refer the reader to the introductions of \cite{BG2} and \cite{AB} for a more detailed account, the following is a list of some of the most relevant works: \cite{Bk},  \cite{BCa}, \cite{AA}, \cite{BK}, \cite{CK},  \cite{EF}, \cite{EFV}, \cite{EV},  \cite{GL1}, \cite{GL2},  \cite{Ku}, \cite{Ku2},  \cite{KL}, \cite{Me},  \cite{Ru}, \cite{Ve},  \cite{Zhu1} and \cite{Zhu2}.

The present paper is organised as follows. In section \ref{s:n}, we introduce the relevant notation and gather some  known results that are needed in our work. In section \ref{s:main} we prove Theorem \ref{main}. The key ingredients in the proof are:
\begin{itemize}
\item[1)] a delicate quantitative version of the Carleman estimate in \cite{EFV}. This result is Lemma \ref{carleman} below, and it represents one of the crucial novelties  of the present work; 
\item[2)] the monotonicity in time in Lemma \ref{mon} below. This result provides a sharp  quantitative passage of the vanishing order information to $t=0$, in which we make explicit the precise dependence on the $C^{1}$ norm of $b$ and $V$.
\end{itemize}
Once Lemmas \ref{carleman} and \ref{mon} are established, we can essentially  repeat the  arguments in \cite{AB} to obtain the information claimed in Theorem \ref{main}. As an application of Theorem \ref{main}, in Section \ref{dof} we prove Theorem \ref{maindf} .
Finally, in Section \ref{landis} we prove Theorems \ref{main2} and \ref{main1}. In  closing, whenever we say  that a constant $N$ is universal, it means that  it depends only on the dimension $n$, and the ellipticity and the Lipschitz constants $\Lambda$ and $M$ in \eqref{ell}, \eqref{ass}. Throughout the paper we will use $N$ as an all purpose constant which may vary from line to line, but will be universal. 	


\section{Preliminary results}\label{s:n}	

In this section we introduce the relevant notation and collect some results that will be used in the main body of the paper.
Points in $\Rn$ will be denoted by $x, y$, etc. For those in space-time $\Rn \times [0, \infty)$, we will use the notation  $X=(x,t), Y=(y,s)$, etc. Whenever convenient, the partial derivative $\partial_{x_i} f$  will be denoted by $f_i$ or $D_i f$, that in $t$ will be denoted by $f_t$ or by $\p_t f$. We will often write $\nabla f$ and  $\operatorname{div} f$ instead of  $\nabla_x f$ and $ \operatorname{div}_x f$, respectively.  
We recall that the \emph{vanishing order} of a function $u$ at $x$ is the largest  integer $\ell$ such that $D^{\alpha} u (x)=0$ for all $|\alpha| \leq \ell$. Here, with $\mathbb N_0 = \mathbb N \cup \{0\}$, we have denoted by $\alpha=(\alpha_1, ..., \alpha_n)\in \mathbb N_0^n$ a multi-index, and have let  $D^\alpha u(x) = \frac{\p^{\alpha_1+...+\alpha_n} u}{\p x_1^{\alpha_1}...\p x_n^{\alpha_n}}(x)$. Given an open set $\Om\subset \Rn\times \R$,
we indicate with $C_0^{\infty}(\Omega)$ the set of functions in $C^\infty(\Om)$ having compact support in $\Omega$. Also, we will indicate by $dX = dxdt$ the Lebesgue measure in $\Rn\times \R$. We will denote by $B_r(x)$ the ball of radius $r$ centred at $x \in \R^n$, whereas $Q_r(x,t)$ will denote the space-time  cylinder $B_r(x) \times [t,t+r^2]$. When $x = 0$, we will simply write $B_r$ and $Q_r$, instead of $B_r(0)$ and $Q_r(0,0)$, respectively. 

Recall now the parabolic dilations $\delta_\la X = (\la x,\la^2 t)$. A function $f:\Rn\times \R\to \R$ is (parabolically) homogeneous of degree $\kappa\in \R$ if
\[
f(\delta_\la X) = \la^\kappa f(X).
\]
The infinitesimal generator of the dilations $\delta_\la$ is the vector field $Z$ specified by the formula  
\begin{align}\label{defz}
Z f(X) = \frac{d}{d\la} f(\delta_\la X)_{\big|_{\la = 1}} =  \langle x, \nabla f \rangle + 2t \p_t f.
\end{align}
Euler formula holds: a $f\in C^1(\Rn\times \R)$ is homogeneous of degree $\kappa$ if and only if 
\begin{equation}\label{euler}
Z f(X) = \kappa f(X).
\end{equation} 
Since $\operatorname{div} Z = n+2$, it is easily verified that for any $b\not= 0$ one has in $\Rn\times (0,\infty)$
\begin{equation}\label{divZb}
\operatorname{div}(t^b Z) = (2b + n + 2) t^b.
\end{equation}
It follows from \eqref{divZb} that
\begin{equation}\label{divzero}
b = - \frac{n}2 - 1\ \ \Longrightarrow\ \  \operatorname{div}(t^b Z) = 0.
\end{equation}
In this paper we routinely identify a vector field in $\Rn$ with the corresponding first-order differential operator. Thus, 
$\mathscr Z = (\mathscr Z_1,...,\mathscr Z_n)$ is identified with $\mathscr Z = \sum_{k=1}^n \mathscr Z_k D_k$. 
With this identification, given a function $f$, we will write $\mathscr Zf = \sum_{k=1}^n \mathscr Z_k D_k f = \sa \mathscr Z,\nabla f\da$.
Given two vector fields $X = \sum_{i=1}^n X_i D_i$, $Y = \sum_{j=1}^n Y_j D_j$, their commutator is the vector field, or first-order differential operator, defined by
\[
[X,Y] =  XY - YX = \sum_{i,j=1}^n \big(X_i D_iY_j - Y_i D_iX_j\big) D_j.
\]
Similar considerations for vector fields in $\R^{n+1}$. For later use, 
we note the following integration by parts formula. Let $\Om\subset \R^{n+1}$ be a  piecewise $C^1$ bounded domain, and $f, g\in C^1(\overline \Om)$, and let $\mathscr Z\in C^{0,1}(\overline \Om,\R^{n+1})$. Then
\begin{equation}\label{ibp}
\int_\Om f \mathscr Z g dX = \int_{\p \Om} f g \sa \mathscr Z,\nu\da d\sigma - \int_{\Om} f g \operatorname{div} \mathscr Z dX - \int_{\Om} g \mathscr Z f dX.
\end{equation}
We will also use the following well-known Rellich identity, see \cite{Rel}.

\begin{prop}\label{P:rellich}
Let $\Om\subset \Rn$ be a piecewise $C^1$  bounded open set, $\mathscr Z\in C^{0,1}(\overline \Om;\Rn)$, and $f\in
C^{1,1}(\overline \Om)$. Then,
 \begin{align*}
 & 2 \int_{\p \Om} \mathscr Z f\ \frac{\p f}{\p \nu} d\sigma -  \int_{\p
\Om} |\nabla
f|^2 \langle \mathscr Z,\nu\rangle d\sigma \\
& = 2 \int_\Om  \mathscr Z f \Delta f dx -  \int_\Om \operatorname{div} \mathscr Z |\nabla
f|^2 dx + 2 \sum_{j=1}^n \int_\Om [D_j, \mathscr Z] f\ D_j f dx. \notag
\end{align*}
\end{prop}
In particular, when $\mathscr Z(x) = x$, then $[D_j, \mathscr Z] = D_j$, and  Proposition \ref{P:rellich} gives for $f\in
C^{1,1}_0(\overline \Om)$,
\begin{equation}\label{rell}
\int_\Om  \sa x,\nabla f\da \Delta f dx = \left(\frac n2 -1\right) \int_\Om |\nabla f|^2 dx.
\end{equation}

The following generalisation of Proposition \ref{P:rellich} for variable coefficient operators will be needed, see \cite[formula (2.4)]{PW}, or also  \cite[Lemma 2.11]{GV} for an extension. In the next statement, $\Om$, $\mathscr Z$ and $f$ are as in Proposition \ref{P:rellich}, and we use the summation convention over repeated indices.

\begin{prop}\label{P:vrellich}
Let $B \in C^{0,1}(\overline \Om;M_{n \times n}(\R))$. Then
 \begin{align*}
 & 2 \int_{\p \Om} \mathscr Z f\ \langle B \n f, \nu \rangle d\sigma -  \int_{\p
\Om} \langle B \n f, \n f \rangle \langle \mathscr Z,\nu\rangle d\sigma = 2 \int_\Om  \mathscr Z f \D( B \n f) dx 
\\
& -  \int_\Om \operatorname{div} \mathscr Z \langle B \n f, \n f \rangle dx + 2  \int_\Om  b_{ij} D_i \mathscr Z_k D_j f D_k f  dx -\int_{\Om} \mathscr Z b_{ij} D_i f D_j f  dx. 
\end{align*}
\end{prop}
In particular, when $\mathscr Z = Bx$ and $B(0)=\mathbb I_n$, one easily obtains 
\[
\begin{cases}
\D \mathscr Z  =n+O(|x|),
\\
\langle \mathscr Z B \n f, \n f \rangle = O(|x|) |\n f|^2,
\\
b_{ij} D_i \mathscr Z_k D_j f D_k f = \langle B  \n f, \n f \rangle +O(|x|)|\n f|^2.
\end{cases}
\]
Therefore, Proposition \ref{P:vrellich} gives for $f \in C^{1,1}_{0}(\overline{\Om}),$
\begin{align}\label{vrell}
\int_\Om  \langle B x, \n f \rangle \D( B \n f) dx= \left(\frac{n}2-1 \right) \int_{\Om} \langle B \n f, \n f \rangle dx + \int_{\Om}O(|x|)|\n f|^2 dx.
\end{align}

Henceforth in this paper, we will denote by $G$ the function in $\Rn\times (0,\infty)$
\begin{equation}\label{G}
G(x,t) = t^{-n/2} e^{-|x|^2/4t}.
\end{equation}
As it is well-known, it satisfies the equations
\begin{equation}\label{nG}
\n G = - \frac{x}{2t} G,\ \ \ \ \ \ \ \Delta G = G_t.
\end{equation}
Furthermore, $G$ is homogeneous of degree $- n$, and therefore according to \eqref{euler} we have in $\Rn\times (0,\infty)$
\begin{equation}\label{ZG}
Z G = - n\ G.
\end{equation}
Keeping \eqref{nG} in mind, 
formula \eqref{ZG} can be rewritten in the form
\begin{equation}\label{Gt}
G_t = - \frac{n}{2t} G - \sa \frac{x}{2t},\n G\da = \left(\frac{|x|^2}{4t^2} -\frac{n}{2t}\right) G.
\end{equation}

We now state four preparatory results that will be  needed in our work in Section \ref{s:main}. In the order, they correspond to \cite[Lemmas 4 \& 5]{EF}, and to \cite[Lemmas 3 \& 4]{EFV}. We consider the function $\theta:(0,1)\to \R^+$ defined by 
\begin{equation}\label{theta}
\theta(t)=t^{1/2}\left( \operatorname{log}\frac{1}{t}\right)^{3/2}.
\end{equation} 
This function corresponds to the choice $\beta = 1$ in \cite[Lemma 5]{EF}, and it does satisfy the assumptions
\[
0 \le \theta(t) \le N,\ 0\le t \le 1,\ \ \ \ \ \text{and}\ \ \ \ \int_0^1 \left(1+\log \frac 1t\right) \theta(t) \frac{dt}t \le N,
\]
in their Lemma 4, except the third one: $|t\theta'(t)|\le N \theta(t)$, since 
\[
\frac{t |\theta'(t)|}{\theta(t)}\ \underset{t\to 1^-}{\longrightarrow}\ \infty.
\]
On the other hand, this small inconsistency does not affect the validity of \cite[Lemmas 4 \& 5]{EF} since what really matters in these results is what happens near $t=0$. For this reason, we will henceforth restrict our focus to an interval which stays away from $t = 1$. For instance, we will confine the domain of the function in \eqref{theta} to the interval $t\in [0,1/2]$, and state our results accordingly. We leave it to the reader to check that there exists $N>0$ such that
\[
|t\theta'(t)|\le N \theta(t),\ \ \ \ \ \ \ \ \ \ t\in (0,1/2].
\]
	
\begin{lemma}\label{sig}
Let $\theta:[0,1/2]\to \R^+$ be as in \eqref{theta}, and for given $\lambda >0$, let $\s:[0,1/2]\to \R$ be the function
\[
\s(t) = t \exp\left[-\int_0^{\la t} \left(1 - \exp\left(-\int_0^s \theta(\tau) \frac{d\tau}\tau\right)\right)\frac{ds}s\right].
\]
Then $\s$ solves the Cauchy problem 
\begin{align}\label{ode}
- \frac{d}{dt}\operatorname{log} \left(\frac{t\s'}{\s}\right) = \frac{d}{dt}\operatorname{log} \left(\frac{\s}{t\s'}\right)=\frac{\theta(\lambda t)}{t},\;\;\;\;\s(0)=0,\;\;\s'(0)=1,
\end{align}
Moreover, there exist a universal constant $N>0$ such that when $0\le \lambda t \le 1/2$:
\begin{itemize}
\item[(i)] $t/N \le \s(t) \le t$;
\item[(ii)]  $1/N \le \s'(t) \le 1$.
\end{itemize}
\end{lemma}
The following consequence of (i) and (ii) in Lemma \ref{sig} will play a crucial role in the proof of Lemma \ref{carl}
\begin{equation}\label{ic}
N^{-1} \le \frac{t\s'(t)}{\s(t)}\le N,\ \ \ \ \ \ \ \ \ \  0\le t\le \frac{1}{2\la}.
\end{equation} 
We also note that \eqref{ode} implies 
\begin{equation}\label{ode2}
- \frac{d}{dt}\left(\frac{t \s'(t)}{\s(t)}\right) = \frac{t \s'(t)}{\s(t)} \frac{\theta(\lambda t)}{t}.
\end{equation}
This identity will be critical in passing from \eqref{I23bis} to \eqref{I23f} below. We remark that, since $0\le \la t\le 1/2$, by our choice of $\theta$ we have $\theta(\la t) \le N$.
The next result is \cite[Lemma 5]{EF}. Since, as mentioned above, we have taken $\beta = 1$ in the definition \eqref{theta}, the same choice has been kept in the statement of Lemma \ref{logi} below. Also, the reader should notice that we have used $2\alpha$, instead of $\alpha$. This makes no difference in the application of the lemma itself. 
	
\begin{lemma}\label{logi}
Let $\theta(t)$ and $\s$ be as in Lemma \ref{sig}. For a given $\alpha >1$ and $\delta \in (0,1)$, let  $\lambda = \A/ \delta^2$. Then there exists a constant $N>0$, depending only on $n$, such that the following inequalities hold for all functions $w \in C_0^{\infty}(\R^n \times [0,1/2\lambda))$,
\begin{align*}
& \int_{\R^n \times [0,1/2\lambda)} \s^{-2\A}\left(\frac{|x|}{t}+\frac{|x|^3}{\A t^2}\right)w^2G dX\le N e^{2\A N}\lambda^{2\A+N}\int_{\R^n \times [0,1/2\lambda)} w^2dX
\\
& + N \delta \int_{\R^n \times [0,1/2\lambda)} \s^{-2\A}\frac{\theta(\lambda t)}{t}w^2 G dX,
\end{align*}
and 
\begin{align*}
& \int_{\R^n \times [0,\frac{1}{2\lambda})}	\s^{-2\A+1}\left(\frac{|x|}{t}+\frac{|x|^3}{\A t^2}+\frac{|x|^2}{\delta t}\right)|\n w|^2G dX 
\\
& \le N e^{2\A N} \lambda^{2\A+N}\int_{\R^n \times [0,\frac{1}{2\lambda})} t|\n w|^2dX + N \delta \int_{\R^n \times [0,\frac{1}{2\lambda})} \s^{-2\A+1}\frac{\theta(\lambda t)}{t}|\n w|^2 G dX.
\end{align*}
\end{lemma}

The following integration by parts lemma will be useful.

\begin{lemma}\label{simple}
Let $\Psi\in C^1([0,\infty))$ with $\Psi\ge 0$, and $h\in C^1_0(\Rn)$. Then for any $\ve>0$ we have
\[
- \int_{\Rn} |x|^2 \left(\ve\Psi(|x|^2) + 2 \Psi'(|x|^2)\right) h^2 dx \le n \int_{\Rn} \Psi(|x|^2) h^2 dx + \frac 1{\ve} \int_{\Rn} \Psi(|x|^2) |\n h|^2 dx.
\]
\end{lemma}

\begin{proof}
For every $r>0$ define $H(r) = \int_{S_r} h^2 d\sigma$. A simple calculation shows
\begin{equation}\label{H'}
H'(r) = \frac{n-1}r H(r) + \frac 2r \int_{S_r} h\sa \nabla h,x\da d\sigma.
\end{equation}
Cavalieri's principle, integration by parts and \eqref{H'}, give
\begin{align*}
& - 2 \int_{\Rn} |x|^2 \Psi'(|x|^2) h^2 dx =  - \int_0^\infty \frac{d}{dr} (\Psi(r^2)) r H(r) dr
\\
& = \int_0^\infty \Psi(r^2) \frac{d}{dr}(r H(r)) dr =  \int_0^\infty \Psi(r^2) \left[H(r) + r H'(r)\right] dr 
\\
& =  n \int_0^\infty \Psi(r^2) H(r) dr + 2 \int_0^\infty \Psi(r^2) \int_{S_r} h\sa \nabla h,x\da d\sigma dr
\\
& = n \int_{\Rn} \Psi(|x|^2) h^2 dx + 2 \int_{\Rn} \Psi(|x|^2) h \sa \nabla h,x\da dx
\\
& \le n \int_{\Rn} \Psi(|x|^2) h^2 dx + \ve \int_{\Rn} |x|^2 \Psi(|x|^2) h^2 dx + \frac 1{\ve} \int_{\Rn} \Psi(|x|^2) |\n h|^2 dx, 
\end{align*}
where in the last step we have used the simple numerical inequality $2AB\le \e A^2 + \e^{-1} B^2$. The desired conclusion follows.

\end{proof}

If for $a>0$ we choose $\Psi(t) = \frac 14 e^{-\frac{t}{4a}}$ in Lemma \ref{simple}, then with $\ve= \frac{1}{4a}$ we obtain the following inequality, see \cite[Lemma 3]{EFV}. We will need it to estimate certain spatial boundary integrals produced by a Lipschitz perturbation of the principal part of \eqref{meq}, see \eqref{bder} below. 

\begin{lemma}\label{logb}
For all $h \in C_0^{\infty}(\R^n)$ one has
$$\int_{\Rn} \frac{|x|^2}{8a}h^2 e^{-\frac{|x|^2}{4a}}dx \le 2a\int_{\Rn} |\n h|^2  e^{-\frac{|x|^2}{4a}} dx +\frac{n}{2}\int_{\Rn} h^2  e^{-\frac{|x|^2}{4a}} dx.$$
\end{lemma}
	
The following result is \cite[Lemma 4]{EFV}, and it follows from Lemma \ref{logb}. It will be used to obtain the quantitative  space-like doubling property in \eqref{doi} below.
	
\begin{lemma}\label{do}
Let $h \in C_{0}^{\infty}(\R^n)$. Assume that $N$ and $\Theta$ verify $N\log(N\Theta) \ge 1$, and that the following inequality holds for $a \le \frac{1}{12 N\log(N\Theta)}$
\begin{align*}
2a \int_{\Rn} |\n h|^2 e^{-\frac{|x|^2}{4a}} dx + \frac{n}{2}\int_{\Rn} h^2 e^{-\frac{|x|^2}{4a}} dx \le N\log(N\Theta)  \int_{\Rn} h^2  e^{-\frac{|x|^2}{4a}} dx.
\end{align*}
Then one has for $0 < r \le 1/2$ 
\begin{align}
\int_{B_{2r}}h^2 dx \le (N\Theta)^N  \int_{B_{r}}h^2 dx. 
\end{align}
\end{lemma}
	
Finally, we record the following standard regularity estimate for solutions to \eqref{meq} which can be found in \cite[Chapter 6]{Li}.

\begin{lemma}\label{len}
Let $u$ be a solution of \eqref{meq} in $Q_4$. Then there exists a constant $D>0$, depending on $n$, $\Lambda$ in \eqref{ell}, and on the constant $M$ in \eqref{ass}, such that 
\begin{align}\label{eq:len}
||u||^2_{L^{\infty}(Q_2)}+||\n u||^2_{L^{\infty}(Q_2)} \le D (1+||V||_{\infty} +||b||^2_{\infty})|| u||^2_{L^2(Q_3)}.
\end{align}
\end{lemma}
	
 \vskip 0.2in


\section{A quantitative Carleman estimate}\label{S:quantcarl}

This section is devoted to the proof of a quantitative $L^2$ Carleman estimate, Theorem \ref{carleman} below, which represents the most technical part of the present work. Our approach is new, and purely based on some carefully chosen vector fields in identities of Rellich type, see Propositions \ref{P:rellich} and \ref{P:vrellich} above. In particular, this allows to avoid spectral gap inequalities. The proof is quite long, but it is our hope that the reader's comprehension will be facilitated, rather than overburdened, by the illustration of the various critical passages. 
Combined with the  monotonicity-in-time result in Lemma \ref{mon} below, Theorem \ref{carleman} will lead to the desired control \eqref{df} of the vanishing order. 
It can be regarded as a quantitative version of the Carleman estimate in \cite[Lemma 6]{EFV}, with the delicate new feature being the asymptotic control \eqref{quanta} below of the parameter $\alpha$ (in the weight $\s_a^{2\alpha}$). It is this bound that captures the sharp dependence of the vanishing order of the solution on the $C^{1}$-norm of the lower-order terms (drift and potential). For the reader not fully familiar with the subtleties of the problem we mention that, as we have explained in the introduction, the linear dependence on $||b||_{1,1/2}$ in \eqref{df} is sharp. To achieve such optimal control we must delicately exploit some additional crucial cancellations that lead to the optimal linear dependence in \eqref{quanta}. Had we, instead, directly used the Carleman estimate in \cite[Lemma 6]{EFV}, we would have obtained a non-optimal quadratic dependence $||b||_{1,1/2}^2$ in \eqref{quanta}, and consequently a similar one in \eqref{df}.

 In what follows, with $\s$ defined as in Lemma \ref{sig}, and $G$ as in \eqref{G}, for any given number $a\in [0,1/2]$ we let $\s_a(t)=\s(t+a)$ and  $G_a(x,t)=G(x,t+a)$. One should keep in mind that the domain of $\s_a$ is $[-a,\frac 12 - a]$ and that $G_a$ is supported in the half-space $t>-a$. 
	
\begin{thrm}\label{carleman}
Let $A(0,0)=\mathbb I_n$. There exist universal constants $N >1$ and $\delta \in (0,1)$ such that, if 
\begin{equation}\label{quanta}
\A \ge N(1+||b||_{1,1/2}+||V||_{1,1/2}^{1/2}),
\end{equation}  
and $\lambda=\A/\delta^2$, then the following inequality holds for all $w \in C_{0}^{\infty}\left(B_4 \times [0,\frac{1}{4\lambda})\right)$ and $0<a\le \frac{1}{4\lambda}$	
\begin{align}\label{carl}
& \A^2 \int_{B_4 \times [0,\frac{1}{4\lambda}) }\s_a^{-2\A}w^2G_adX +\A \int_{B_4 \times [0,\frac{1}{4\lambda}) } \s_a^{1-2\A}|\n w|^2G_adX\\
& \le N \int_{B_4 \times [0,\frac{1}{4\lambda}) } \s_a^{1-2\A}\left[w_t + \D(A(x,t) \n w)+\langle b(x,t), \n w \rangle+V(x,t)w\right]^2 G_adX
\notag\\
& +N^{2\A}\A^{2\A}\underset{t \ge 0}{\operatorname{sup}}\int_{\Rn} \left[w^2
+|\n w|^2\right] dx
\notag\\
&+\s(a)^{-2\A}\left(-\frac{a}{N}\int_{\Rn} |\n w (x,0)|^2G(x,a)dx + N \A \int_{\Rn} w^2(x,0)G(x,a)dx\right).
\notag
\end{align}  
\end{thrm}
	
\begin{proof}
For the reasons that we have explained above, the proof of this result is quite technical. To facilitate its presentation, we brake it into \textbf{Steps 1-4}, with \textbf{Steps 1 \& 2} being the most technical and novel parts:
\begin{itemize}
\item In \textbf{Step 1} we establish the Carleman estimate \eqref{carl} when $A(x,t) \equiv \mathbb I_n$ and the lower-order coefficients are time-independent, i.e., $b(x,t) =  b(x)$ and $V(x,t) = V(x)$. This part contains the core of the ideas, and consequently we have provided full details.
\item In \textbf{Step 2} we prove \eqref{carl} when $A(x,t)=A(x)$,  $b(x,t) =  b(x)$, $V(x,t) = V(x)$. As the reader will see, passing from the Laplacian to a time-independent, variable coefficient matrix, requires a considerable amount of technical work. 
\item In \textbf{Step 3} we prove \eqref{carl} when $A(x,t)$ is general, but the lower-order coefficients are still time-independent, i.e. $b(x,t) =  b(x)$, $V(x,t) = V(x)$. 
\item In \textbf{Step 4} we finally remove the restriction of being time-independent on $b$ and $V$.
\end{itemize}

\vskip 0.2in

\noindent \textbf{Step 1: $\boxed{A\equiv \mathbb I_n$,  $b(x,t) =  b(x)$, $V(x,t) = V(x)}.$} 
\\
We observe preliminarily that, if for a given $w \in C_{0}^{\infty}\left(B_4 \times [0,\frac{1}{4\lambda})\right)$, we make the change of variable $(x,t)\to (x,\tau)$ with $\tau = t + a$ in the two integrals on the set $B_4 \times [0,\frac{1}{4\lambda})$ in \eqref{carl}, then the domain of integration becomes $B_4 \times [a,a+\frac{1}{4\lambda})$, whereas the integrand now involves the functions $\s(t)$ in Lemma \ref{sig}, $G(x,t)$ in \eqref{G}, and $w_a(x,t)=w(x, t-a)$. Since $a + \frac 1{4\la} \le \frac 1{2\la}$, it is clear that, by renaming $w$ the function $w_a$, then (under the present assumptions on $A(x,t), b(x,t)$ and $V(x,t)$) \eqref{carl} is equivalent to proving that for every $w \in C_0^{\infty}\left(B_4 \times [a,\frac{1}{2 \lambda})\right)$, one has   
\begin{align}\label{carletto}
& \A^2 \int_{B_4 \times [a,\frac{1}{2\lambda}) }\s^{-2\A}w^2 G dX +\A \int_{B_4 \times [a,\frac{1}{2\lambda}) } \s^{1-2\A}|\n w|^2 G dX\\
& \le N \int_{B_4 \times [a,\frac{1}{2\lambda}) } \s^{1-2\A}\left[w_t + \Delta w+\langle b(x), \n w \rangle+V(x)w\right]^2 G dX
\notag\\
& +N^{2\A}\A^{2\A}\underset{t \ge a}{\operatorname{sup}}\int_{\Rn} \left[w^2
+|\n w|^2\right] dx
\notag\\
&+\s(a)^{-2\A}\left(-\frac{a}{N}\int_{\Rn} |\n w (x,a)|^2G(x,a)dx + N \A \int_{\Rn} w^2(x,a)G(x,a)dx\right).
\notag
\end{align} 
The reader should pay attention to the fact that, since we are denoting by $w$ the function $w_a$, the traces in the third integral in the right-hand side involve $w_a(x,a) = w(x,0)$, where $w$ is the original cut-off function supported in the time interval $[0,\frac{1}{4\lambda})$.
With this being said, to proceed with the proof of the Carleman estimate \eqref{carletto}, we perform an important conjugation, and rename  the function $w$ as follows
\[
w(x,t) = t^\beta \s(t)^{\A} \w(x,t) G(x,t)^{-1/2},
\]
or equivalently, 
\begin{equation}\label{tw}
\w(x,t)^2 = t^{-2\beta} \s(t)^{-2\A} w(x,t)^2 G(x,t),
\end{equation}
where the crucial parameter $\beta\not= 0$ will be suitably chosen subsequently (we will take $\beta = - \frac n4$, see \eqref{divzero2}, \eqref{beta} below). One might wonder why, in this conjugation, we are not just choosing $w(x,t) = \s(t)^{\A} \w(x,t) G(x,t)^{-1/2}$, as in the integrand in the left-hand side of \eqref{carletto}, but this seemingly more natural choice would not produce some delicate cancellations (e.g., see \eqref{I23} below) that are hidden in such estimate. With this being said, with $\w$ as in \eqref{tw},     
a standard computation gives
\begin{align*}
& w_t + \Delta w+\langle b(x), \n w \rangle+V(x)w = \left(\tilde w_t + \Delta \tilde w\right)t^\beta \s^\alpha G^{-1/2} + \tilde w\left[(t^\beta \s^\alpha G^{-1/2})_t + t^\beta \s^\alpha \Delta(G^{-1/2})\right]
\\
& + 2 t^\beta \s^\alpha \sa \n \tilde w,\n(G^{-1/2})\da+ t^\beta \s^\alpha \langle b(x), \n \tilde w \rangle G^{-1/2} + t^\beta \s^\alpha \tilde w \langle b(x), \n(G^{-1/2})\rangle  +t^\beta \s^\alpha G^{-1/2} V(x)\tilde w
\\
& = \bigg\{\tilde w_t + \Delta \tilde w 
+ \left(\alpha \frac{\s'}{\s} +\frac \beta{t}- \frac 12 \frac{G_t}{G} +\frac{\Delta(G^{-1/2})}{G^{-1/2}}\right)\tilde w - \sa\n \tilde w,\frac{\n G}{G}\da  - \frac 12 \tilde w \sa b(x),\frac{\n G}{G}\da \\
& + \langle b(x), \n \tilde w \rangle + V(x) \tilde w\bigg\} t^\beta \s^\alpha G^{-1/2}.
\end{align*}
Keeping \eqref{nG} in mind, we have
\[
\tilde w_t - \sa\n \tilde w,\frac{\n G}{G}\da = \frac{1}{2t} Z \tilde w,
\]
where $Z$ is defined by \eqref{defz}. 
Also, an application of the chain rule, combined with \eqref{nG} and \eqref{Gt}, gives
\[
\frac{\Delta(G^{-1/2})}{G^{-1/2}} - \frac 12 \frac{G_t}{G} = \frac 34 \frac{|\n G|^2}{G^2} - \frac{G_t}{G} = \frac{n}{2t} - \frac{|x|^2}{16 t^2}.
\]
We thus find
\begin{align*}
& \s^{1-2\alpha}\left(w_t + \Delta w+\langle b(x), \n w \rangle+V(x)w\right)^2 G = t^{2\beta} \s\ \bigg\{\frac{1}{2t} Z \tilde w + + \langle b(x), \n \tilde w \rangle
\\
&  + \Delta \tilde w + \left(\alpha \frac{\s'}{\s} + \frac{n+2\beta}{2t} - \frac{|x|^2}{16 t^2} + \frac{1}{4t} \sa b(x),x\da + V(x)\right) \tilde w\bigg\}^2
\\
& \ge N^{-1} t^{1+2\beta} \bigg\{\frac{1}{2t} Z \tilde w +  \langle b(x), \n \tilde w \rangle + \Delta \tilde w
\\
&   + \left(\alpha \frac{\s'}{\s} + \frac{n+2\beta}{2t} - \frac{|x|^2}{16 t^2} + \frac{1}{4t} \sa b(x),x\da + V(x)\right) \tilde w\bigg\}^2
\\
& \ge N^{-3/2} t^{1+2\beta} \left(\frac{t \s'}{\s}\right)^{-1/2} \bigg\{\frac{1}{2t} Z \tilde w +  \langle b(x), \n \tilde w \rangle
\\
&  + \Delta \tilde w + \left(\alpha \frac{\s'}{\s} + \frac{n+2\beta}{2t} - \frac{|x|^2}{16 t^2} + \frac{1}{4t} \sa b(x),x\da + V(x)\right) \tilde w\bigg\}^2,
\end{align*}
where in the first inequality in the right-hand side we have used (i) in Lemma \ref{sig}, which gives $t/N \le \s(t) \le t$ for $0\le t \le \frac{1}{2\la}$, whereas in the latter we have used \eqref{ic}. The reader might wonder now why we have forced the term $\left(\frac{t \s'}{\s}\right)^{-1/2}$ to appear. Later in the computations, after integrating by parts, we will need to extract a delicate coercivity from various terms. This is where the factor $\left(\frac{t \s'}{\s}\right)^{-1/2}$ will play a critical role, through the ODE \eqref{ode}. We need to introduce it now since, once we start splitting the right-hand side in various components, its present positivity will be lost. To continue estimating from below, we next set
\begin{equation}\label{A}
\mathcal A = \frac{1}{2t} Z \tilde w +  \langle b(x), \n \tilde w \rangle
\end{equation}
and 
\begin{equation}\label{B}
\mathcal B = \Delta \tilde w + \left(\alpha \frac{\s'}{\s} + \frac{n+2\beta}{2t} - \frac{|x|^2}{16 t^2} + \frac{1}{4t} \sa b(x),x\da + V(x)\right) \tilde w.
\end{equation}
The importance of  choosing $\mathcal A$ as in \eqref{A} will become evident in the forthcoming analysis. We now use the obvious inequality 
\[
\left(\mathcal A + \mathcal B\right)^2 \ge \mathcal A^2 + 2 \mathcal A \mathcal B,
\]
to obtain
\begin{align*}
& \int_{B_4 \times [a,\frac{1}{2\lambda}) } \s^{1-2\A}\left[w_t + \Delta w+\langle b(x), \n w \rangle+V(x)w\right]^2 G dX
\\
& \ge N^{-3/2} \int_{B_4 \times [a,\frac{1}{2\lambda}) } t^{1+2\beta} \left(\frac{t \s'}{\s}\right)^{-1/2} \left[\frac{1}{2t} Z \tilde w +  \langle b(x), \n \tilde w \rangle\right]^2 dX
\\
& + 2 N^{-3/2} \int_{B_4 \times [a,\frac{1}{2\lambda}) } t^{1+2\beta} \left(\frac{t \s'}{\s}\right)^{-1/2} \left[\frac{1}{2t} Z \tilde w +  \langle b(x), \n \tilde w \rangle\right]
\\
& \times \left[\Delta \tilde w + \left(\alpha \frac{\s'}{\s} + \frac{n+2\beta}{2t} - \frac{|x|^2}{16 t^2} + \frac{1}{4t} \sa b(x),x\da + V(x)\right) \tilde w\right] dX.
\end{align*}
We next split the right-hand side in the latter inequality into several pieces, and then treat each piece separately\footnote{To simplify the notation, we will from now on drop the reference to the domain of integration in all the integrals  supported in $B_4 \times [a, \frac{1}{2\lambda}]$}. For notational convenience, we move the universal constant $N^{-3/2}$ to the left-hand side of the above inequality, obtaining
\begin{align*}
& N^{3/2} \int \s^{1-2\A}\left[w_t + \Delta w+\langle b(x), \n w \rangle+V(x)w\right]^2 G dX
\\
& \ge  \int t^{1+2\beta} \left(\frac{t \s'}{\s}\right)^{-1/2} \left[\frac{1}{2t} Z \tilde w +  \langle b(x), \n \tilde w \rangle\right]^2 dX  +  \int t^{2\beta} \left(\frac{t \s'}{\s}\right)^{-1/2} Z \tilde w \Delta \w dX
\notag\\
& + \int t^{2\beta} \left(\frac{t \s'}{\s}\right)^{-1/2} \w Z\w \left(\alpha \frac{\s'}{\s} + \frac{n+2\beta}{2t} - \frac{|x|^2}{16 t^2} + \frac{1}{4t} \sa b(x),x\da + V(x)\right) dX
\notag\\
& + 2 \int t^{1+2\beta} \left(\frac{t \s'}{\s}\right)^{-1/2} \langle b(x), \n \tilde w \rangle \Delta \w dX
\notag \\
& + 2 \int t^{1+2\beta} \left(\frac{t \s'}{\s}\right)^{-1/2} \w \langle b(x), \n \tilde w \rangle \left(\alpha \frac{\s'}{\s} + \frac{n+2\beta}{2t} - \frac{|x|^2}{16 t^2} + \frac{1}{4t} \sa b(x),x\da + V(x)\right) dX
\\
& = \mathcal I_1 + \mathcal I_2 + \mathcal I_3 + \mathcal I_4 + \mathcal I_5.
\end{align*}
We begin by analysing $\mathcal I_3$, since it is in such term that we decide how to choose the value of $\beta$. Observe that 
\[
Z\left\{t^{2\beta}\left(\frac{n+2\beta}{2t} - \frac{|x|^2}{16 t^2}\right)\right\} = (4\beta- 2) \left\{t^{2\beta}\left(\frac{n+2\beta}{2t} - \frac{|x|^2}{16 t^2}\right)\right\}.
\]
This just follows from the fact that the function within curly brackets in the left-hand side is homogeneous of degree $\kappa = 4\beta- 2$ with respect to the parabolic dilations. It thus ensues from \eqref{divzero} that in $\Rn\times [a,\infty)$
\begin{equation}\label{divzero2}
\operatorname{div}\left\{t^{2\beta}\left(\frac{n+2\beta}{2t} - \frac{|x|^2}{16 t^2}\right)Z \right\} = 0,
\end{equation}
provided that
\begin{equation}\label{beta}
4\beta- 2 + n + 2 = 0\ \Longleftrightarrow\ \beta = - \frac n4.
\end{equation}
Henceforth, with $\beta$ fixed as in \eqref{beta}, we let $\mathscr Z = t^{2\beta}\left(\frac{n+2\beta}{2t} - \frac{|x|^2}{16 t^2}\right) Z$, so that \eqref{divzero2} gives $\operatorname{div} \mathscr Z = 0$. With this choice, by an application of \eqref{ibp} with $\Om = B_4 \times (a,\frac{1}{2 \lambda})$, 
 $g = \w^2/2$ and $f = \left(\frac{t \s'}{\s}\right)^{-1/2}$, we obtain for 
 the following terms in the second integral in the right-hand side of the above inequality
\begin{align*}
& \int t^{2\beta} \left(\frac{t \s'}{\s}\right)^{-1/2} \w Z\w \left(\frac{n+2\beta}{2t} - \frac{|x|^2}{16 t^2}\right) dX 
 = \frac 12 \int  \left(\frac{t \s'}{\s}\right)^{-1/2} \mathscr Z \w^2 dX
\\
& =  \int_{B_4\times\{a\}} \w^2   t^{-\frac n2+1}\left(\frac{|x|^2}{16 t^2} - \frac{n}{4t}\right) \left(\frac{t \s'}{\s}\right)^{-1/2} dx
\\
& +  \int   \w^2 \ t^{-\frac n2 + 1}\left(\frac{|x|^2}{16 t^2} - \frac{n}{4t}\right) \frac{d}{dt} \left(\frac{t \s'}{\s}\right)^{-1/2} dX.
\end{align*}
Next, with the following different choice of $\mathscr Z = t^{-\frac n2 -1} Z$, we have from \eqref{divzero} and \eqref{ibp}, applied with $g = \w^2$ and $f = \left(\frac{t \s'}{\s}\right)^{1/2}$, 
\begin{align*}
& \int t^{2\beta} \left(\frac{t \s'}{\s}\right)^{-1/2} \w Z\w\ \alpha \frac{\s'}{\s} dX = \frac{\alpha}2 \int \left(\frac{t \s'}{\s}\right)^{1/2} \mathscr Z \w^2 dX
\\
& = - \alpha \int_{B_4\times\{a\}} \w^2 t^{-\frac n2} \left(\frac{t \s'}{\s}\right)^{1/2}  dx - \alpha \int \w^2 t^{-\frac n2} \frac{d}{dt} \left(\frac{t \s'}{\s}\right)^{1/2} dX.
\end{align*}
Combining terms, we thus find
\begin{align}\label{I3}
\mathcal I_3 = & \int t^{- \frac n2} \left(\frac{t \s'}{\s}\right)^{-1/2} \w Z\w \left(\frac{1}{4t} \sa b(x),x\da + V(x)\right) dX
\\
& +  \int   \w^2 \ t^{-\frac n2}\left(\frac{|x|^2}{16 t} - \frac{n}{4}\right) \frac{d}{dt} \left(\frac{t \s'}{\s}\right)^{-1/2} dX - \alpha \int  \w^2\ t^{-\frac n2} \frac{d}{dt} \left(\frac{t \s'}{\s}\right)^{1/2} dX
\notag\\
& +  \int_{B_4\times\{a\}} \w^2   t^{-\frac n2}\left(\frac{|x|^2}{16 t} - \frac{n}{4}\right) \left(\frac{t \s'}{\s}\right)^{-1/2} dx
 - \alpha \int_{B_4\times\{a\}} \w^2 t^{-\frac n2} \left(\frac{t \s'}{\s}\right)^{1/2}  dx.
\notag
\end{align}
Next, using \eqref{rell} we obtain
\begin{align*}
\mathcal I_2 & = \int t^{-\frac n2} \left(\frac{t \s'}{\s}\right)^{-1/2} \sa x,\n \tilde w\da \Delta \w dX + 2 \int t^{-\frac n2 +1} \left(\frac{t \s'}{\s}\right)^{-1/2}  \tilde w_t \Delta \w dX
\\
& = \left(\frac n2 - 1\right) \int t^{-\frac n2} \left(\frac{t \s'}{\s}\right)^{-1/2} |\n \w|^2 dX -  \int t^{-\frac n2 +1} \left(\frac{t \s'}{\s}\right)^{-1/2} \frac{d}{dt}(|\n \w|^2) dX
\\
& = \left(\frac n2 - 1\right) \int t^{-\frac n2} \left(\frac{t \s'}{\s}\right)^{-1/2} |\n \w|^2 dX +  \int \frac{d}{dt}(t^{-\frac n2 +1}) \left(\frac{t \s'}{\s}\right)^{-1/2} |\n \w|^2 dX 
\\
& + \int t^{-\frac n2 +1} \frac{d}{dt}\left(\frac{t \s'}{\s}\right)^{-1/2} |\n \w|^2 dX + \int_{B_4\times\{a\}} t^{-\frac n2 +1} \left(\frac{t \s'}{\s}\right)^{-1/2} |\n \w|^2 dx
\\
& = \int t^{-\frac n2 +1} \frac{d}{dt}\left(\frac{t \s'}{\s}\right)^{-1/2} |\n \w|^2 dX + \int_{B_4\times\{a\}} t^{-\frac n2 +1} \left(\frac{t \s'}{\s}\right)^{-1/2} |\n \w|^2 dx.
\end{align*}
Next, we apply Proposition \ref{P:rellich}, with the choice $\mathscr Z(x) = \sum_{k=1}^n b_k(x) D_k$, to find 
\begin{align*}
|\mathcal I_4| & =  \bigg|2 \int t^{-\frac n2 + 1} \left(\frac{t \s'}{\s}\right)^{-1/2} \langle b(x), \n \tilde w \rangle \Delta \w dX\bigg| \le \bigg|\int t^{-\frac n2 + 1} \left(\frac{t \s'}{\s}\right)^{-1/2} \operatorname{div} b\ |\n \w|^2 dX\bigg|
\\
& + 2 \int t^{-\frac n2 + 1} \left(\frac{t \s'}{\s}\right)^{-1/2} \big(\sum_{j,k=1} (D_j b_k)^2\big)^{1/2}  |\n \w|^2 dX
\\
& \le 3 ||b||_{1,1/2} \int t^{-\frac n2 + 1} \left(\frac{t \s'}{\s}\right)^{-1/2} |\n \w|^2 dX.
\end{align*}
We infer
\begin{align*}
\mathcal I_4 & \ge - 3 ||b||_{1,1/2} \int t^{-\frac n2 + 1} \left(\frac{t \s'}{\s}\right)^{-1/2} |\n \w|^2 dX.
\end{align*}
Finally, we analyse 
\begin{align*}
\mathcal I_5 & =  \int t^{-\frac n2+1} \left(\frac{t \s'}{\s}\right)^{-1/2} \langle b(x), \n(\tilde w^2) \rangle \left(\alpha \frac{\s'}{\s} + \frac{n}{4t} - \frac{|x|^2}{16 t^2} + \frac{1}{4t} \sa b(x),x\da + V(x)\right) dX
\\
& = - \alpha \int \w^2  t^{-\frac n2} \left(\frac{t \s'}{\s}\right)^{1/2} \operatorname{div} b\  dX - \frac n4 \int \w^2  t^{-\frac n2} \left(\frac{t \s'}{\s}\right)^{-1/2} \operatorname{div} b\  dX
\\
& + \frac{1}{16} \int \w^2  t^{-\frac n2-1} \left(\frac{t \s'}{\s}\right)^{-1/2} \big(|x|^2 \operatorname{div} b + 2\sa x,b\da\big)\  dX
\\
& - \frac{1}{4} \int \w^2  t^{-\frac n2} \left(\frac{t \s'}{\s}\right)^{-1/2} \big(|b|^2 + x_k b_j D_j b_k\big)  dX - \int \w^2  t^{-\frac n2+1} \left(\frac{t \s'}{\s}\right)^{-1/2} \big(V \operatorname{div} b + \sa b,\n V\da\big)  dX.
\end{align*}
This gives for some universal $C>0$
\begin{align*}
\mathcal I_5 & \ge - \alpha  ||b||_{1,1/2} \int \w^2  t^{-\frac n2} \left(\frac{t \s'}{\s}\right)^{1/2} dX - \frac n4 ||b||_{1,1/2} \int \w^2  t^{-\frac n2} \left(\frac{t \s'}{\s}\right)^{-1/2} dX
\\
& - C ||b||_{1,1/2} \int \w^2 t^{-\frac n2} \frac{|x|}t  \left(\frac{t \s'}{\s}\right)^{-1/2} dX - C ||b||_{1,1/2}^2 \int \w^2  t^{-\frac n2} \left(\frac{t \s'}{\s}\right)^{-1/2}  dX
\\
& - C ||V||_{1,1/2} ||b||_{1,1/2} \int \w^2  t^{-\frac n2+1} \left(\frac{t \s'}{\s}\right)^{-1/2}  dX.
\end{align*}
To proceed with the proof of \eqref{carl}, we next undo the conjugation \eqref{vtw}. Keeping \eqref{beta} in mind, we observe that
\begin{equation}\label{nw}
|\n \w|^2 = t^{\frac n2} \s^{-2\A} G \left\{|\n w|^2 + w^2 \frac{|x|^2}{16t^2} - \frac{1}{4t} \sa x,\n w^2\da\right\},
\end{equation}
and use this identity to express in terms of $|\n w|^2$ the integrals in $\mathcal I_2$ and $\mathcal I_4$ in which $|\n \w|^2$ appears. We begin with the former
\begin{align*}
\mathcal I_2 & =  \int t^{-\frac n2 +1} \frac{d}{dt}\left(\frac{t \s'}{\s}\right)^{-1/2} |\n \w|^2 dX + \int_{B_4\times\{a\}} t^{-\frac n2 +1} \left(\frac{t \s'}{\s}\right)^{-1/2} |\n \w|^2 dx
\\
& = \int t \s^{-2\A} \frac{d}{dt}\left(\frac{t \s'}{\s}\right)^{-1/2}  |\n w|^2 G  dX + \int \s^{-2\A} \frac{d}{dt}\left(\frac{t \s'}{\s}\right)^{-1/2}  w^2 \frac{|x|^2}{16t} G dX
\\
& - \frac{1}{4} \int  \s^{-2\A} \frac{d}{dt}\left(\frac{t \s'}{\s}\right)^{-1/2} \sa x,\n w^2\da G dX + \int_{B_4\times\{a\}} t \s^{-2\A}  \left(\frac{t \s'}{\s}\right)^{-1/2} |\n w|^2 G dx
\\
& + \int_{B_4\times\{a\}} \s^{-2\A}  \left(\frac{t \s'}{\s}\right)^{-1/2} w^2 \frac{|x|^2}{16t} G dx - \frac{1}{4} \int_{B_4\times\{a\}}  \s^{-2\A} \left(\frac{t \s'}{\s}\right)^{-1/2} \sa x,\n w^2\da G dx.
\end{align*}
Integrating in $x\in B_4$, we obtain 
\begin{align}\label{tricky}
& - \frac{1}{4} \int  \sa x,\n w^2\da G dx = \frac{n}{4} \int  w^2 G dx + \frac{1}{4} \int w^2 \sa x,\n G\da dx
\\
& = \frac{n}{4} \int  w^2 G dx - \int w^2 \frac{|x|^2}{8t}  G dx.
\notag
\end{align}
Substituting \eqref{tricky} in the above expression of $\mathcal I_2$, we find
\begin{align}\label{I2}
\mathcal I_2 & = \int t \s^{-2\A} \frac{d}{dt}\left(\frac{t \s'}{\s}\right)^{-1/2}  |\n w|^2 G  dX - \int \s^{-2\A} \frac{d}{dt}\left(\frac{t \s'}{\s}\right)^{-1/2}  w^2 \left(\frac{|x|^2}{16t} -\frac n4\right)G dX
\\
& + \int_{B_4\times\{a\}} t \s^{-2\A}  \left(\frac{t \s'}{\s}\right)^{-1/2} |\n w|^2 G dx - \int_{B_4\times\{a\}} \s^{-2\A} \left(\frac{t \s'}{\s}\right)^{-1/2}   \left(\frac{|x|^2}{16t} -\frac n4\right) w^2 G dx.
\notag
\end{align}
From \eqref{tw} and \eqref{I3} we have
\begin{align}\label{I33}
\mathcal I_3 & = \frac 12 \int t^{- \frac n2} \left(\frac{t \s'}{\s}\right)^{-1/2} Z(\w^2) \left(\frac{1}{4t} \sa b(x),x\da + V(x)\right) dX
\\
& +  \int  \s^{-2\A} \frac{d}{dt} \left(\frac{t \s'}{\s}\right)^{-1/2} w^2  \left(\frac{|x|^2}{16 t} - \frac{n}{4}\right)  G dX - \alpha \int \s^{-2\A}  \frac{d}{dt} \left(\frac{t \s'}{\s}\right)^{1/2} w^2 G dX
\notag\\
& +  \int_{B_4\times\{a\}} \s^{-2\A}  t^{-\frac n2}\left(\frac{|x|^2}{16 t} - \frac{n}{4}\right) \left(\frac{t \s'}{\s}\right)^{-1/2} w^2 G dx
 - \alpha \int_{B_4\times\{a\}} \s^{-2\A}  t^{-\frac n2} \left(\frac{t \s'}{\s}\right)^{1/2} w^2 G dx.
\notag
\end{align}
Combining \eqref{I2} with \eqref{I33}, after two crucial cancellations, we obtain
\begin{align}\label{I23}
\mathcal I_2 + \mathcal I_3 & = \int t \s^{-2\A} \frac{d}{dt}\left(\frac{t \s'}{\s}\right)^{-1/2}  |\n w|^2 G  dX - \alpha \int \s^{-2\A}  \frac{d}{dt} \left(\frac{t \s'}{\s}\right)^{1/2} w^2 G dX
\\
& + \int_{B_4\times\{a\}} t \s^{-2\A}  \left(\frac{t \s'}{\s}\right)^{-1/2} |\n w|^2 G dx - \alpha \int_{B_4\times\{a\}} \s^{-2\A}  t^{-\frac n2} \left(\frac{t \s'}{\s}\right)^{1/2} w^2 G dx
\notag
\\
& + \frac 12 \int t^{- \frac n2} \left(\frac{t \s'}{\s}\right)^{-1/2} Z(\w^2) \left(\frac{1}{4t} \sa b(x),x\da + V(x)\right) dX.
\notag
\end{align}

In a similar fashion, using \eqref{tricky} again, we obtain  
\begin{align}\label{I4}
\mathcal I_4 & \ge - 3 ||b||_{1,1/2} \int t \s^{-2\A} \left(\frac{t \s'}{\s}\right)^{-1/2} |\n w|^2 G dX - \frac{3n}4 ||b||_{1,1/2} \int  \s^{-2\A} \left(\frac{t \s'}{\s}\right)^{-1/2}  w^2 G dX,
\end{align}
and also 
\begin{align}\label{I5}
\mathcal I_5 & \ge - \alpha  ||b||_{1,1/2} \int \s^{-2\A} \left(\frac{t \s'}{\s}\right)^{1/2} w^2 G dX - \frac n4 ||b||_{1,1/2} \int \s^{-2\A} \left(\frac{t \s'}{\s}\right)^{-1/2} w^2 G dX
\\
& - C ||b||_{1,1/2} \int \s^{-2\A} \frac{|x|}t  \left(\frac{t \s'}{\s}\right)^{-1/2} w^2 G dX - C ||b||_{1,1/2}^2 \int \s^{-2\A} \left(\frac{t \s'}{\s}\right)^{-1/2} w^2 G  dX
\notag
\\
& - C ||V||_{1,1/2} ||b||_{1,1/2} \int \s^{-2\A}  t \left(\frac{t \s'}{\s}\right)^{-1/2} w^2 G  dX.
\notag
\end{align}
At this point, we are left with unraveling the last term, which still involves $\w$, in the right-hand side of \eqref{I23} above. Applying \eqref{ibp} with $g = \w^2$, $f = \left(\frac{t \s'}{\s}\right)^{-1/2} \sa b,x\da$, $\mathscr Z = t^{-\frac n2-1} Z$, for which $\operatorname{div}\mathscr Z = 0$ (see \eqref{divzero}), and then using \eqref{tw}, we find
\begin{align*}
& \frac 12 \int t^{- \frac n2} \left(\frac{t \s'}{\s}\right)^{-1/2} Z(\w^2) \frac{1}{4t} \sa b(x),x\da  dX = \frac 18 \int \left(\frac{t \s'}{\s}\right)^{-1/2} \sa b,x\da \mathscr Z(\w^2) dX
\\
& = - \frac 14 \int_{B_4\times\{a\}}  \s^{-2\A} \left(\frac{t \s'}{\s}\right)^{-1/2} \sa b,x\da w^2 G dx - \frac 14 \int  \s^{-2\A} \frac{d}{dt}\left(\frac{t \s'}{\s}\right)^{-1/2} \sa b,x\da w^2 G dX
\\
& - \frac 18 \int  \s^{-2\A} \left(\frac{t \s'}{\s}\right)^{-1/2} \frac{D_i b_j x_i x_j + \sa b,x\da}{t} w^2 G dX 
\\
& \ge  - C ||b||_{1,1/2} \int_{B_4\times\{a\}}  \s^{-2\A} \left(\frac{t \s'}{\s}\right)^{-1/2} w^2 G dx - \frac 14 \int  \s^{-2\A} \frac{d}{dt}\left(\frac{t \s'}{\s}\right)^{-1/2} \sa b,x\da w^2 G dX
\\
& - C ||b||_{1,1/2} \int \s^{-2\A} \frac{|x|}t  \left(\frac{t \s'}{\s}\right)^{-1/2} w^2 G dX,
\end{align*}
for a universal $C>0$. In a similar fashion, the part of the integral involving $V(x)$ gives instead
\begin{align*}
& \frac 12 \int t^{- \frac n2} \left(\frac{t \s'}{\s}\right)^{-1/2} Z(\w^2) V(x) dX \ge - C ||V||_{1,1/2} \int_{B_4\times\{a\}} t \s^{-2\A} \left(\frac{t \s'}{\s}\right)^{-1/2} V w^2 G dx
\\
& - \int t \s^{-2\A} \frac{d}{dt}\left(\frac{t \s'}{\s}\right)^{-1/2} V w^2 G dX -C ||V||_{1,1/2} \int t \s^{-2\A} \left(\frac{t \s'}{\s}\right)^{-1/2} w^2 G dX. 
\end{align*}
Substituting the latter two estimates in \eqref{I23}, we obtain
\begin{align}\label{I23bis}
\mathcal I_2 + \mathcal I_3 & \ge \int t \s^{-2\A} \frac{d}{dt}\left(\frac{t \s'}{\s}\right)^{-1/2}  |\n w|^2 G  dX - \alpha \int \s^{-2\A}  \frac{d}{dt} \left(\frac{t \s'}{\s}\right)^{1/2} w^2 G dX
\\
& - \frac 14 \int  \s^{-2\A} \frac{d}{dt}\left(\frac{t \s'}{\s}\right)^{-1/2} \sa b,x\da w^2 G dX - \int t \s^{-2\A} \frac{d}{dt}\left(\frac{t \s'}{\s}\right)^{-1/2} V w^2 G dX
\notag\\
& - C ||b||_{1,1/2} \int \s^{-2\A} \frac{|x|}t  \left(\frac{t \s'}{\s}\right)^{-1/2} w^2 G dX -C ||V||_{1,1/2} \int t \s^{-2\A} \left(\frac{t \s'}{\s}\right)^{-1/2} w^2 G dX
\notag\\
& + \int_{B_4\times\{a\}} t \s^{-2\A}  \left(\frac{t \s'}{\s}\right)^{-1/2} |\n w|^2 G dx - \alpha \int_{B_4\times\{a\}} \s^{-2\A} \left(\frac{t \s'}{\s}\right)^{1/2} w^2 G dx
\notag
\\
&   - C ||b||_{1,1/2} \int_{B_4\times\{a\}}  \s^{-2\A} \left(\frac{t \s'}{\s}\right)^{-1/2} w^2 G dx 
  - C ||V||_{1,1/2} \int_{B_4\times\{a\}} t \s^{-2\A} \left(\frac{t \s'}{\s}\right)^{-1/2}  w^2 G dx. 
\notag
\end{align}
To proceed with the proof of \eqref{carletto}, we use in a crucial way the differential equation \eqref{ode2}, which we substitute in \eqref{I23bis}, obtaining 
\begin{align}\label{I23f}
\mathcal I_2 + \mathcal I_3 & \ge \frac12 \int t \s^{-2\A} \left(\frac{t \s'}{\s}\right)^{-1/2} \frac{\theta(\lambda t)}{t} |\n w|^2 G  dX + \frac{\alpha}{2} \int \s^{-2\A}  \left(\frac{t \s'}{\s}\right)^{1/2} \frac{\theta(\lambda t)}{t} w^2 G dX
\\
& - \frac 18 \int  \s^{-2\A} \left(\frac{t \s'}{\s}\right)^{-1/2} \frac{\theta(\lambda t)}{t} \sa b,x\da w^2 G dX - \frac 12 \int t \s^{-2\A} \left(\frac{t \s'}{\s}\right)^{-1/2} \frac{\theta(\lambda t)}{t} V w^2 G dX
\notag\\
& - C ||b||_{1,1/2} \int \s^{-2\A} \frac{|x|}t  \left(\frac{t \s'}{\s}\right)^{-1/2} w^2 G dX -C ||V||_{1,1/2} \int t \s^{-2\A} \left(\frac{t \s'}{\s}\right)^{-1/2} w^2 G dX
\notag\\
& + \int_{B_4\times\{a\}} t \s^{-2\A}  \left(\frac{t \s'}{\s}\right)^{-1/2} |\n w|^2 G dx - \alpha \int_{B_4\times\{a\}} \s^{-2\A} \left(\frac{t \s'}{\s}\right)^{1/2} w^2 G dx
\notag
\\
&   - C ||b||_{1,1/2} \int_{B_4\times\{a\}}  \s^{-2\A} \left(\frac{t \s'}{\s}\right)^{-1/2} w^2 G dx 
- C ||V||_{1,1/2} \int_{B_4\times\{a\}} t \s^{-2\A} \left(\frac{t \s'}{\s}\right)^{-1/2}  w^2 G dx. 
\notag
\end{align}
We now combine \eqref{I4} and \eqref{I5} with \eqref{I23f}. We also use the bounds \eqref{ic}, to eliminate the weights $\left(\frac{t\s'}{\s}\right)^{\mp 1/2}$ from all the integrals involved. We thus find
\begin{align}\label{c1}
 &N^{3/2} \int \s^{1-2\A}\left[w_t + \Delta w+\langle b(x), \n w \rangle+V(x)w\right]^2 G dX
 \\
& \ge \int \left(\frac{\theta(\lambda t)}{Nt} -N ||b||_{1,1/2}\right) t \s^{-2\A}  |\n w|^2 G  dX
\notag
\\
& + \int \left\{\A \frac{\theta(\lambda t)}{Nt} -N\left(1+\A+ ||b||_{1,1/2} \right) ||b||_{1,1/2}\right\} \s^{-2\A}   w^2 G dX
\notag
\\
& - N ||b||_{1,1/2} \int  \left(\theta (\lambda t) + 1 \right) \s^{-2\A} \frac{|x|}{t}  w^2 G dX
\notag
\\
& - N ||V||_{1,1/2} \int  (\theta (\lambda t)+1)   \s^{-2\A}   w^2 G dX- N  ||b||_{1,1/2} ||V||_{1,1/2} \int  t \s^{-2\A}   w^2 G dX
\notag
\\
&+ \frac{1}{N} \int_{B_4\times\{a\}} t \s^{-2\A}   |\n w|^2 G dx - N \int_{B_4\times\{a\}}(\A + ||b||_{1,1/2} + t ||V||_{1,1/2} ) \s^{-2\A}  w^2 G dx.\notag
\end{align}
Keeping in mind that $\theta(\lambda t) \le N$, we next use Lemma \ref{logi} to estimate the third integral in the right-hand side of \eqref{c1}, obtaining 
\begin{align}\label{c2}
& \left|\int  \left(\theta (\lambda t) + 1 \right) \s^{-2\A} \frac{|x|}{t}  w^2 G dX\right| \le N\int   \s^{-2\A} \frac{|x|}{t}  w^2 G dX
\\
&\le N e^{2\A N}\lambda^{2\A+N}\int w^2dX
+ N \delta \int \s^{-2\A}\frac{\theta(\lambda t)}{t}w^2 G dX.
\notag
\end{align}
Inserting \eqref{c2} into \eqref{c1}, we find
\begin{align}\label{c3}
& N^{3/2} \int \s^{1-2\A}\left[w_t + \Delta w+\langle b(x), \n w \rangle+V(x)w\right]^2 G dX 
\\
& \ge \int \left(\frac{\theta(\lambda t)}{Nt}  -N ||b||_{1,1/2}\right) t \s^{-2\A}  |\n w|^2 G  dX
\notag
\\
& + \int \left\{\left(\frac{\A}{N}- N \delta ||b||_{1,1/2} \right) \frac{\theta(\lambda t)}{t} -N\left(1+\A+ ||b||_{1,1/2} \right) ||b||_{1,1/2} \right\} \s^{-2\A}   w^2 G dX
\notag\\
& -N ||b||_{1,1/2} e^{2\A N}\lambda^{2\A+N}\int w^2dX - N ||V||_{1,1/2} \int \left(1 + t ||b||_{1,1/2} \right) \s^{-2\A}   w^2 G dX
\notag
\\
&+ \frac{1}{N} \int_{B_4\times\{a\}} t \s^{-2\A}   |\n w|^2 G dx - N \int_{B_4\times\{a\}}(\A + ||b||_{1,1/2} + t ||V||_{1,1/2} ) \s^{-2\A}  w^2 G dx.
\notag
\end{align}
Now for $\A \ge 2N^2 (1+||b||_{1,1/2}+||V||_{1,1/2}^{1/2})$ and $\delta \in (0,1)$ (to be chosen later in \eqref{delta}), we observe that for $a \le t \le \frac{1}{2 \lambda}= \frac{\delta^2}{2\A}$, the following bounds hold:
\begin{equation}\label{crcbd}
\begin{cases}
\left(\frac{\A}{N}- N \delta ||b||_{1,1/2} \right) \frac{\theta(\lambda t)}{t} -N\left(1+\A+ ||b||_{1,1/2} \right) ||b||_{1,1/2}  \ge \frac{\A}{2N} \frac{\theta(\lambda t)}{t} - \A^2,
\\
\left(1 +  t ||b||_{1,1/2} \right)||V||_{1,1/2} \le \left(1 +  \frac{\delta^2}{2\A} \ \frac{\A}{2N^2} \right)\frac{\A^2}{4N^4}  \ \le \frac{\A^2}{N},
\\
\A + ||b||_{1,1/2} + t ||V||_{1,1/2}  \le \A + \frac{\A}{2N^2} +\frac{\A^2}{4N^4} \  \frac{\delta^2}{2\A} \  \le 2 \A.
\end{cases}
\end{equation}
Using \eqref{crcbd}, along with $\frac{1}{N} t \leq\s(t) \leq t$ (see (i) after Lemma \ref{sig}), we find from \eqref{c3} 
\begin{align}\label{csf}
& N^{3/2} \int \s^{1-2\A}\left[w_t + \Delta w+\langle b(x), \n w \rangle+V(x)w\right]^2 G dX
\\
& \ge \int \left(\frac{\theta(\lambda t)}{Nt}  -\frac{\A}{2N}\right)  \s^{1-2\A}  |\n w|^2 G  dX
+ \int \left( \frac{\A}{2N} \frac{\theta(\lambda t)}{t} - 2 \A^2 \right) \s^{-2\A}   w^2 G dX
\notag
\\
& -N \A e^{2\A N}\lambda^{2\A+N}\int w^2dX 
+ \frac{1}{N} \int_{B_4\times\{a\}} t \s^{-2\A}   |\n w|^2 G dx - 2 N \A \int_{B_4\times\{a\}} \s^{-2\A}  w^2 G dx.
\notag
\end{align}
To control from below the first two integrals in the right-hand side of \eqref{csf}, we make critical use of the following consequence of \eqref{theta}: for $0<\lambda t \le 1/2$, we have
\begin{align}\label{theta1}
\frac{\theta(\lambda t)}{t} \ge \frac{\lambda^{1/2}}{t^{1/2}} (\log 2)^{3/2}  \geq \frac{\lambda}{N} = \frac{\alpha}{ N\delta^2}.
\end{align} 
We thus find that for $\A \ge 2N^2 (1+||b||_{1,1/2}+||V||_{1,1/2}^{1/2})$
 \begin{align*}
 & N^{3/2} \int \s^{1-2\A}\left[w_t + \Delta w+\langle b(x), \n w \rangle+V(x)w\right]^2 G dX
 \notag
 \\
 & \ge \int \left( \frac{\alpha}{N\delta^2}  -\frac{\A}{2N}\right)  \s^{1-2\A}  |\n w|^2 G  dX
+ \int \left( \frac{\A^2 }{2N \delta^2} - 2 \A^2  \right) \s^{-2\A}   w^2 G dX
 \notag
 \\
 & -N \A e^{2\A N}\lambda^{2\A+N}\int w^2dX 
 + \frac{1}{N} \int_{B_4\times\{a\}} t \s^{-2\A}   |\n w|^2 G dx - 2 N \A \int_{B_4\times\{a\}} \s^{-2\A}  w^2 G dx.
 \end{align*}
At this point we choose $\delta \in (0,1)$ such that
\begin{align}\label{delta}
& \frac{1}{2N \delta^2} - 2 \ge 1.
\end{align}
For  
\[
\A \ge 2N^2 (1+||b||_{1,1/2}+||V||_{1,1/2}^{1/2}),
\]
and $\delta \in (0,1)$ as in \eqref{delta}, we finally obtain
\begin{align*}
& N^{3/2} \int \s^{1-2\A}\left[w_t + \Delta w+\langle b(x), \n w \rangle+V(x)w\right]^2 G dX
\notag
\\
& \ge \frac{\A}{2N} \int  \s^{1-2\A}  |\n w|^2 G  dX
+ \A^2  \int \s^{-2\A}   w^2 G dX
\notag
\\
& -N \A e^{2\A N}\lambda^{2\A+N}\int w^2dX 
+ \frac{1}{N} \int_{B_4\times\{a\}} t \s^{-2\A}   |\n w|^2 G dx - 2 N \A \int_{B_4\times\{a\}} \s^{-2\A}  w^2 G dx.
\end{align*}
\\
This completes the proof of Theorem \ref{carleman}, when $A\equiv \mathbb I_n$,  $b(x,t) =  b(x)$, $V(x,t) = V(x).$
We now proceed with the proof of the general case, which we subdivide into three steps.

\vskip 0.2in
	
\noindent \textbf{Step 2: $\boxed{A(x,t)=A(x)$,  $b(x,t) =  b(x)$, $V(x,t) = V(x)}.$}

We follow the strategy of the case $A \equiv \mathbb I_n$ in \textbf{Step 1}, and with $\s$ as in Lemma \ref{sig}, consider 
\begin{align}\label{vtw}
\w(x,t) =t^{ \frac n4} \s^{- \A} w(x,t) G(x,t)^{1/2},
\end{align} 
see \eqref{tw} and \eqref{beta}. Then, $w(x,t)= t^{-\frac n4 } \s^{\A}\w(x,t) G(x,t)^{-1/2}$ and we find (to simplify the notation, we henceforth write $A = A(x), b = b(x)$ and $V = V(x)$) 
\begin{align*}
& w_t + \D (A  \n w) +\langle b, \n w \rangle+V w
\\
& = \left(\tilde w_t + \D ( A \n \tilde w)\right)t^{-\frac n4} \s^\alpha G^{-1/2} + \tilde w\left[(t^{-\frac n4} \s^\alpha G^{-1/2})_t + t^{-\frac n4} \s^\alpha \D(A \n G^{-1/2})\right]
\\
& + 2 t^{-\frac n4} \s^\alpha \sa A \n \tilde w,\n(G^{-1/2})\da+ t^{-\frac n4} \s^\alpha \langle b, \n \tilde w \rangle G^{-1/2} + t^{-\frac n4} \s^\alpha \tilde w \langle b, \n(G^{-1/2})\rangle  + t^{-\frac n4} \s^\alpha G^{-1/2} V \tilde w
\\
& = \bigg\{\tilde w_t + \D (A \n \tilde w ) 
+ \left(\alpha \frac{\s'}{\s} -\frac n{4t}- \frac 12 \frac{G_t}{G} +\frac{\D ( A \n G^{-1/2})}{G^{-1/2}}\right)\tilde w - \sa A \n \tilde w,\frac{\n G}{G}\da  - \frac 12 \tilde w \sa b,\frac{\n G}{G}\da \\
& + \langle b, \n \tilde w \rangle + V \tilde w\bigg\} t^{-\frac n4} \s^\alpha G^{-1/2}.
\end{align*}
By direct computations (see \eqref{nG} and \eqref{Gt}), we have 
\begin{align}\label{a} 
& \frac{\D( A \n G^{-1/2})}{G^{-1/2}} - \frac 12 \frac{G_t}{G} = \frac{\D ( A x )}{4t} + \frac{\sa Ax, x \da}{16t^2} - \frac{|x|^2}{8t^2} + \frac{n}{4t}
\\
&=  \frac{n}{2t} - \frac{|x|^2}{8t^2} + \frac{\sa Ax, x \da}{16t^2} +\frac{1}{4t}(\D(Ax)-n).
\notag
\end{align}
We now introduce the vector field 
\begin{align}\label{defzv}
\F f \overset{def}=  \langle Ax, \n f\rangle + 2t \p_t f,
\end{align}
which is the variable-coefficient version of $Z$ in \eqref{defz}.
Using Lemma \ref{sig} and \eqref{ic}, we  find 
\begin{align*}
& \s^{1-2\alpha}\left(w_t + \D( A \n w) +\langle b, \n w \rangle+V w\right)^2 G 
\\
&= t^{-\frac n2} \s\ \bigg\{\frac{1}{2t} \F \tilde w  + \langle b, \n \tilde w \rangle  + \D( A \n \tilde w)
\\
& + \left(\alpha \frac{\s'}{\s} + \frac{n}{4t} + \frac{\sa Ax, x \da}{16t^2} - \frac{|x|^2}{8t^2} +\frac{1}{4t}(\D(Ax)-n) + \frac{1}{4t} \sa b,x\da + V\right) \tilde w\bigg\}^2
\\
& \ge N^{-3/2} t^{-\frac n2 +1} \left(\frac{t \s'}{\s}\right)^{-1/2}\bigg\{\frac{1}{2t} \F \tilde w  + \langle b, \n \tilde w \rangle + \D( A \n \tilde w) 
\\
&   + \left(\alpha \frac{\s'}{\s} + \frac{n}{4t}  + \frac{\sa Ax, x \da}{16t^2} - \frac{|x|^2}{8t^2} +\frac{1}{4t}(\D(Ax)-n)+ \frac{1}{4t} \sa b,x\da + V\right) \tilde w\bigg\}^2.
\end{align*}
Similarly to the case $A\equiv I_n$, we now use the inequality $(\mathcal A+ \mathcal B)^2 \geq \mathcal A^2 + 2 \mathcal A \mathcal B$, with
$$\mathcal A= \frac{\F\w}{2t}+\langle b, \n \w\rangle, 
$$
and  
$$ \mathcal B =\D( A \n \tilde w) + \left(\alpha \frac{\s'}{\s} + \frac{n}{4t}  + \frac{\sa Ax, x \da}{16t^2} - \frac{|x|^2}{8t^2} +\frac{1}{4t}(\D(Ax)-n)+ \frac{1}{4t} \sa b,x\da + V\right)\w.$$
We find
\begin{align}\label{vcs}
&N^{3/2}\int \s^{1-2\A}\left( w_t + \D(A \n w) +\langle b, \n w\rangle+ V w\right)^2 G dX
\\
&\ge \int t^{-\frac{n}2 + 1} \left(\frac{t \s'}{\s}\right)^{-1/2} \left(\frac{\F \w}{2t}+\langle b, \n \w\rangle\right)^2 dX + \int t^{-\frac{n}2} \left(\frac{t \s'}{\s}\right)^{-1/2} \F \w  \D( A \n \tilde w) dX 
\notag
\\
& + 2 \int t^{-\frac{n}2 + 1} \left(\frac{t \s'}{\s}\right)^{-1/2} \langle b, \n \w\rangle\D(A \n \w) dX + \frac12 \int t^{-\frac{n}2} \left(\frac{t \s'}{\s}\right)^{-1/2} [\D(A x)-n]\left( \frac{ \F \w}{2t} + \langle b, \n \w\rangle\right)\w  dX 
\notag
\\
& + \int t^{-\frac{n}2} \left(\frac{t \s'}{\s}\right)^{-1/2} \left( \frac{\A\s'}{\s} +\frac{n}{4t} - \frac{|x|^2}{8t^2}  + \frac{\sa Ax, x \da}{16t^2}+\frac{1}{4t}\left\langle  b, x\right\rangle+ V \right) \w \F\w dX. 
\notag
\\
&+2 \int t^{-\frac{n}2 +1} \left(\frac{t \s'}{\s}\right)^{-1/2}\left( \frac{\A\s'}{\s} + \frac{n}{4t}-\frac{|x|^2}{8t^2}  +\frac{\sa Ax, x \da}{16t^2}+ \frac{1}{4t}\left\langle  b, x\right\rangle + V  \right)\w \langle b, \n \w\rangle dX  
\notag 
\\
&=\mi_1+\mi_2+\mi_3+\mi_4+\mi_5+\mi_6.
\notag
\end{align}

We first analyse $\mi_5,$ the integral containing $\w \F\w$. Following the reasoning in \eqref{divzero2}, \eqref{beta}, we introduce the vector field
\begin{equation}\label{newzeta}
\mathscr Z = t^{-\frac{n}2} \left[ \frac{n}{4t} - \frac{|x|^2}{8t^2}  + \frac{1}{16t^2}\langle Ax,x\rangle \right] \F, 
\end{equation}
obtaining 
\begin{align*}
\D \mathscr Z &=  t^{-\frac{n}2} \left\{\frac{n}{4t} \D(Ax) - \frac{1}{8t^2}\D(|x|^2 Ax)  + \frac{1}{16t^2}\D(\left\langle Ax,x\right\rangle Ax)\right\}\\
 & \ + \frac{d}{dt}\left(t^{-\frac{n}2} \left\{ \frac{n}{2} - \frac{|x|^2}{4t}  + \frac{1}{8t}\sa Ax, x\da \right\}\right).
\end{align*}    
Keeping in mind that $A = A(x)$, the assumptions \eqref{ass} and $A(0,0) =\mathbb I_n$ give  
\begin{equation}\label{divaa}
\begin{cases}
\D (A(x)x)=n+O(|x|),
\\
\D(|x|^2 Ax)=(n+2)|x|^2 +O(|x|^3),
\\ 
\D(\left\langle Ax,x\right\rangle Ax)=(n+2)|x|^2 +O(|x|^3).
\end{cases}
\end{equation}
Using \eqref{divaa}, and writing $A=A -\mathbb I_n + \mathbb I_n$, we find
\begin{align*}
\D \mathscr Z &= t^{-\frac{n}2} \left\{\frac{n^2}{4t}   -\frac{n+2}{16t^2}|x|^2 + \frac{O(|x|)}{t} + \frac{O(|x|^3)}{t^2}\right\}\\
 & + t^{-\frac{n}2} \left\{ -\frac{n^2}{4t} + \frac{n+2}{8t^2}|x|^2 -\frac{n+2}{16t^2}|x|^2  - \frac{n+2}{16t^2} \sa (A-I)x, x\da \right\}.
\end{align*}
Finally, we use \eqref{ass} to obtain
  \begin{align*}
  \D \mathscr Z &=  t^{-\frac{n}2} \left\{\frac{O(|x|)}{t} + \frac{O(|x|^3)}{t^2}\right\}.
  \end{align*}
Therefore, with $\mathscr Z$ as in \eqref{newzeta}, and an application of \eqref{ibp} with $\Om = B_4 \times (a,\frac{1}{2 \lambda})$, 
 $g = \w^2/2$ and $f = \left(\frac{t \s'}{\s}\right)^{-1/2}$, we deduce the following 
\begin{align*}
&\int t^{-\frac{n}2} \left(\frac{t \s'}{\s}\right)^{-1/2} \left(  \frac{n}{4t} - \frac{|x|^2}{8t^2}  + \frac{\sa Ax,x \da}{16t^2}  \right) \w \F\w dX
\\
&=\int_{B_4 \times \{a\}} \w^2
t^{-\frac{n}2+1} \left(\frac{t \s'}{\s}\right)^{-1/2} \left(  \frac{|x|^2}{8t^2} -\frac{n}{4t} -   \frac{\sa Ax,x \da}{16t^2}  \right)dx
\\
&+\int \w^2
t^{-\frac{n}2+1}  \left(  \frac{|x|^2}{8t^2} -\frac{n}{4t} -   \frac{\sa Ax,x \da}{16t^2} \right) \frac{d}{dt}\left(\frac{t \s'}{\s}\right)^{-1/2}dX\\
&+ \int t^{-\frac{n}2} \left(\frac{t \s'}{\s}\right)^{-1/2}\left(\frac{O(|x|)}{t} + \frac{O(|x|^3)}{t^2}\right) \w^2  dX.
\end{align*}  
Next, with the different choice $\mathscr Z = t^{-\frac{n}2-1}\F$, using \eqref{ass} and $A(0,0) =\mathbb I_n$, which give $\D (A(x)x)=n+O(|x|)$, we have
\begin{align}\label{dvz}
\D \mathscr Z = t^{-\frac{n}2-1} \D(Ax)  -n t^{-\frac{n}2-1} = t^{-\frac{n}2-1}O(|x|). 
\end{align}
Applying again \eqref{ibp} with $\Om = B_4 \times (a,\frac{1}{2 \lambda})$, $g = \w^2/2$ and $f = \left(\frac{t \s'}{\s}\right)^{1/2}$, we thus find 
\begin{align*}
& \int t^{-\frac n2} \left(\frac{t \s'}{\s}\right)^{-1/2} \w \F\w\ \alpha \frac{\s'}{\s} dX = \A \int O(|x|)t^{-\frac{n}2-1} \left(\frac{t \s'}{\s}\right)^{1/2} \w^2 dX
\\
& - \alpha \int_{B_4\times\{a\}} \w^2 t^{-\frac n2} \left(\frac{t \s'}{\s}\right)^{1/2}  dx - \alpha \int \w^2 t^{-\frac n2} \frac{d}{dt} \left(\frac{t \s'}{\s}\right)^{1/2} dX.
\end{align*}
Combining terms, we conclude that
\begin{align}\label{vi5}
\mi_5 &= \int t^{-\frac{n}2} \left( \frac{1}{4t} \sa b(x),x \da + V(x)  \right) \w \F\w dX - \alpha \int \w^2 t^{-\frac n2} \frac{d}{dt} \left(\frac{t \s'}{\s}\right)^{1/2} dX
\\
& +\int \w^2
t^{-\frac{n}2} \left(  \frac{|x|^2}{8t} -\frac{n}{4} -   \frac{\sa Ax, x \da}{16t}  \right) \frac{d}{dt}\left(\frac{t \s'}{\s}\right)^{-1/2}dX
\notag
\\
& + \int_{B_4 \times \{a\}} \w^2
t^{-\frac{n}2} \left(\frac{t \s'}{\s}\right)^{-1/2} \left(  \frac{|x|^2}{8t} -\frac{n}{4} -   \frac{\sa Ax, x \da}{16t} \right)dx - \alpha \int_{B_4\times\{a\}} \w^2 t^{-\frac n2} \left(\frac{t \s'}{\s}\right)^{1/2}  dx
\notag
\\
& + \int t^{-\frac{n}2} \left(\frac{t \s'}{\s}\right)^{-1/2}\left(\frac{O(|x|)}{t} + \frac{O(|x|^3)}{t^2}\right) \w^2  dX + \A \int O(|x|)t^{-\frac{n}2-1} \left(\frac{t \s'}{\s}\right)^{1/2} \w^2 dX.
\notag
\end{align}
Using \eqref{vtw} and \eqref{vi5}, we find
\begin{align}\label{vi5c}
\mi_5 &= \int t^{-\frac{n}2} \left( \frac{1}{4t} \sa b(x),x \da + V(x)  \right) \w \F\w dX - \alpha \int \s^{-2\A} \frac{d}{dt} \left(\frac{t \s'}{\s}\right)^{1/2} w^2 G dX
\\
& +\int \s^{-2\A} \left(\frac{t \s'}{\s}\right)^{-1/2} \left(  \frac{|x|^2}{8t} -\frac{n}{4} -   \frac{\sa Ax, x \da}{16t}  \right) \frac{d}{dt}\left(\frac{t \s'}{\s}\right)^{-1/2} w^2 GdX
\notag
\\
& + \int \s^{-2\A} \left(\frac{t \s'}{\s}\right)^{-1/2}\left(\frac{O(|x|)}{t} + \frac{O(|x|^3)}{t^2}\right) w^2 G  dX + \A \int \frac{O(|x|)}{t} \left(\frac{t \s'}{\s}\right)^{1/2} w^2 G dX
 \notag\\
&+ \int_{B_4 \times \{a\}} \s^{-2\A}
 \left(\frac{t \s'}{\s}\right)^{-1/2} \left(  \frac{|x|^2}{8t} -\frac{n}{4} -   \frac{\sa Ax, x \da}{16t} \right)w^2 Gdx - \alpha \int_{B_4\times\{a\}} \s^{-2\A} \left(\frac{t \s'}{\s}\right)^{1/2} w^2 G  dx.
\notag
\end{align}

We now analyse $\mi_2$. Using \eqref{vrell} with $B=A$ and $f=\w$, and recalling that we are assuming that $A$ be time-independent, we obtain
\begin{align*}
\mi_2 & = \int t^{-\frac n2} \left(\frac{t \s'}{\s}\right)^{-1/2} \sa Ax,\n \tilde w\da \D(A \n \w) dX + 2 \int t^{-\frac n2 +1} \left(\frac{t \s'}{\s}\right)^{-1/2}  \tilde w_t \D(A \n \w) dX
\\
& = \left(\frac n2 - 1\right) \int t^{-\frac n2} \left(\frac{t \s'}{\s}\right)^{-1/2} \sa A\n \w, \n \w\da  dX -  \int t^{-\frac n2 +1} \left(\frac{t \s'}{\s}\right)^{-1/2} \frac{d}{dt}(\sa A\n \w, \n \w\da) dX
\\
&\ + \int O(|x|)t^{-\frac n2} \left(\frac{t \s'}{\s}\right)^{-1/2} |\n \w|^2  dX
\\
& = \left(\frac n2 - 1\right) \int t^{-\frac n2} \left(\frac{t \s'}{\s}\right)^{-1/2} \sa A\n \w, \n \w\da dX +  \int \frac{d}{dt}(t^{-\frac n2 +1}) \left(\frac{t \s'}{\s}\right)^{-1/2}\sa A\n \w, \n \w\da dX 
\\
& + \int t^{-\frac n2 +1} \frac{d}{dt}\left(\frac{t \s'}{\s}\right)^{-1/2} \sa A\n \w, \n \w\da dX + \int_{B_4\times\{a\}} t^{-\frac n2 +1} \left(\frac{t \s'}{\s}\right)^{-1/2} \sa A\n \w, \n \w\da dx
\\
&\ + \int O(|x|)t^{-\frac n2} \left(\frac{t \s'}{\s}\right)^{-1/2} |\n \w|^2  dX
\\
& = \int t^{-\frac n2 +1} \frac{d}{dt}\left(\frac{t \s'}{\s}\right)^{-1/2} \sa A\n \w, \n \w\da dX + \int_{B_4\times\{a\}} t^{-\frac n2 +1} \left(\frac{t \s'}{\s}\right)^{-1/2} \sa A\n \w, \n \w\da dx
\\
&\ + \int O(|x|)t^{-\frac n2} \left(\frac{t \s'}{\s}\right)^{-1/2} |\n \w|^2  dX.
\end{align*}
Observe now that, similarly to \eqref{nw}, we presently have
\[
\sa A\n \w, \n \w \da = t^{\frac n2} \s^{-2\A} G \left\{\sa A\n w, \n w \da  + w^2 \frac{\sa Ax, x\da}{16t^2} - \frac{1}{4t} \sa Ax,\n w^2\da\right\}.
\] 
Proceeding as for \eqref{tricky}, and also keeping in mind that $A(0,0)=\mathbb I_n$, as well as the assumption \eqref{ass}, we obtain
\begin{align*}
-\int_{B_4} \sa Ax,\n w^2\da G dx = n\int_{B_4} w^2 G dx -\int_{B_4} \frac{|x|^2}{2t}w^2Gdx + \int_{B_4} O(|x|)w^2 G dx +\int_{B_4} \frac{O(|x|^3)}{t}w^2 G dx.
\end{align*}
We use this identity in the above expression of $\mi_2$ to express in terms of $\sa A\n w, \n w\da $ the two integrals in which $\sa A\n \w, \n \w \da$ appears. Also, an application of Cauchy-Schwarz inequality to $|\n \w|^2 = t^{\frac n2} \s^{-2\A} G \left|\n w  - w \frac{x}{4t}\right|^2$, gives
\begin{align*}
 t^{-\frac n2} |\n \w|^2 \le 2 \s^{-2\A} |\n w|^2 G + 2 \s^{-2\A}\frac{|x|^2}{16t^2}w^2G.
\end{align*}
We conclude that		
\begin{align}\label{vi2c}
\mi_2 & \ge \int t\s^{-2\A} \frac{d}{dt}\left(\frac{t \s'}{\s}\right)^{-1/2} \sa A\n w, \n w\da G dX
\\
& - \int \s^{-2\A} \frac{d}{dt}\left(\frac{t \s'}{\s}\right)^{-1/2} w^2 \left(\frac{|x|^2}{8t}-\frac{\sa Ax, x\da}{16t}-\frac n4 \right)G dX -N\int \left(|x|+\frac{|x|^3}{t}\right)\s^{-2\A} \left(\frac{t \s'}{\s}\right)^{-1/2} w^2 G dX
\notag
\\
&  -N \int  |x|\s^{-2\A}\left(\frac{t \s'}{\s}\right)^{-1/2} |\n w|^2  dX -N \int  \frac{|x|^3}{t^2}\s^{-2\A}\left(\frac{t \s'}{\s}\right)^{-1/2} w^2 G dX
\notag\\
& + \int_{B_4\times\{a\}} t\s^{-2\A} \left(\frac{t \s'}{\s}\right)^{-1/2} \sa A\n w, \n w\da G dx - \int_{B_4\times\{a\}} \s^{-2\A} \left(\frac{t \s'}{\s}\right)^{-1/2} w^2 \left(\frac{|x|^2}{8t}-\frac{\sa Ax, x\da}{16t}-\frac n4 \right)G dx
\notag
\\
& -N\int_{B_4\times\{a\}} \left(|x|+\frac{|x|^3}{t}\right)\s^{-2\A} \left(\frac{t \s'}{\s}\right)^{-1/2} w^2 G dx. 
\notag
\end{align}
We now combine \eqref{vi2c} with \eqref{vi5c} and, as in the case when $A \equiv \mathbb{I}_n$, we exploit some crucial cancellations, obtaining  
\begin{align}\label{vi25c}
\mi_2 + \mi_5 & \ge \int t\s^{-2\A} \frac{d}{dt}\left(\frac{t \s'}{\s}\right)^{-1/2} \sa A\n w, \n w\da G dX  - \alpha \int \s^{-2\A} \frac{d}{dt} \left(\frac{t \s'}{\s}\right)^{1/2} w^2 G dX
\\
& + \int_{B_4\times\{a\}} t\s^{-2\A} \left(\frac{t \s'}{\s}\right)^{-1/2} \sa A\n w, \n w\da G dx - \alpha \int_{B_4\times\{a\}} \s^{-2\A} \left(\frac{t \s'}{\s}\right)^{1/2} w^2 G  dx
\notag
\\
&\ -N \int |x| \s^{-2\A}\left(\frac{t \s'}{\s}\right)^{-1/2} |\n w|^2  dX -N \int  \frac{|x|^3}{t^2}\s^{-2\A}\left(\frac{t \s'}{\s}\right)^{-1/2} w^2 G dX 
\notag
\\
& -N \A \int \frac{|x|}{t} \s^{-2\A} \left(\frac{t \s'}{\s}\right)^{1/2} w^2 G dX -N\int_{B_4\times\{a\}} \left(|x|+\frac{|x|^3}{t}\right)\s^{-2\A} \left(\frac{t \s'}{\s}\right)^{-1/2} w^2 G dx
\notag
\\
&+ \int t^{-\frac{n}2} \left( \frac{1}{4t} \sa b,x \da + V\right) \w \F\w dX.
\notag
\end{align}	
As a help to the reader, we mention that in \eqref{vi25c} we have  used the fact  that $\A>1,$ $t<1$, and that $|x|<4.$	
To estimate from below the last integral in the right-hand side of \eqref{vi25c}, we apply \eqref{ibp} twice. First, we take $g = \w^2$, $f = \left(\frac{t \s'}{\s}\right)^{-1/2} \sa b,x\da$ and $\mathscr Z = t^{-\frac n2-1} \F$. Noting that, in view of  \eqref{dvz}, this choice gives $\operatorname{div}\mathscr Z = t^{-\frac{n}2-1}O(|x|)$, and then using \eqref{vtw}, that $|x| \le 4$, and also that $A$ is Lipschitz in $x-$variable (see \eqref{ass}), we find for some universal large $N$
\begin{align*}
& \frac 12 \int t^{- \frac n2} \left(\frac{t \s'}{\s}\right)^{-1/2} \F(\w^2) \frac{1}{4t} \sa b,x\da  dX = \frac 18 \int \left(\frac{t \s'}{\s}\right)^{-1/2} \sa b,x\da \mathscr Z(\w^2) dX
\\
& = - \frac 14 \int_{B_4\times\{a\}}  \s^{-2\A} \left(\frac{t \s'}{\s}\right)^{-1/2} \sa b,x\da w^2 G dx - \frac 14 \int  \s^{-2\A} \frac{d}{dt}\left(\frac{t \s'}{\s}\right)^{-1/2} \sa b,x\da w^2 G dX
\\
& - \frac 18 \int  \s^{-2\A} \left(\frac{t \s'}{\s}\right)^{-1/2} \frac{D_i b_j x_j (Ax)_i + \sa b,Ax \da}{t} w^2 G dX + \int \frac{O(|x|)}{t} \s^{-2\A} \left(\frac{t \s'}{\s}\right)^{-1/2} \sa b, x \da w^2 G dX
\\
& \ge  - N ||b||_{1,1/2} \int_{B_4\times\{a\}}  \s^{-2\A} \left(\frac{t \s'}{\s}\right)^{-1/2} w^2 G dx - \frac 14 \int  \s^{-2\A} \frac{d}{dt}\left(\frac{t \s'}{\s}\right)^{-1/2} \sa b,x\da w^2 G dX
\\
& - N ||b||_{1,1/2} \int \s^{-2\A} \frac{|x|}t  \left(\frac{t \s'}{\s}\right)^{-1/2} w^2 G dX.
\end{align*}
In a similar fashion, the part of the integral involving $V$ gives 
\begin{align*}
& \frac 12 \int t^{- \frac n2} \left(\frac{t \s'}{\s}\right)^{-1/2} \F(\w^2) V dX \ge - N ||V||_{1,1/2} \int_{B_4\times\{a\}} t \s^{-2\A} \left(\frac{t \s'}{\s}\right)^{-1/2} w^2 G dx
\\
& - \int t \s^{-2\A} \frac{d}{dt}\left(\frac{t \s'}{\s}\right)^{-1/2} V w^2 G dX -N ||V||_{1,1/2} \int  \s^{-2\A} \left(\frac{t \s'}{\s}\right)^{-1/2} w^2 G dX.
\end{align*}
Using the latter two estimates in \eqref{vi25c}, we obtain 
\begin{align}\label{vi25b}
\mi_2 + \mi_5 & \ge \int t\s^{-2\A} \frac{d}{dt}\left(\frac{t \s'}{\s}\right)^{-1/2} \sa A\n w, \n w\da G dX  - \alpha \int \s^{-2\A} \frac{d}{dt} \left(\frac{t \s'}{\s}\right)^{1/2} w^2 G dX
\\
& - \frac 14 \int  \s^{-2\A} \frac{d}{dt}\left(\frac{t \s'}{\s}\right)^{-1/2} \sa b,x\da w^2 G dX - \int t \s^{-2\A} \frac{d}{dt}\left(\frac{t \s'}{\s}\right)^{-1/2} V w^2 G dX
\notag
\\
&- N ||b||_{1,1/2} \int \s^{-2\A} \frac{|x|}t  \left(\frac{t \s'}{\s}\right)^{-1/2} w^2 G dX -N ||V||_{1,1/2} \int  \s^{-2\A} \left(\frac{t \s'}{\s}\right)^{-1/2} w^2 G
\notag
\\
& + \int_{B_4\times\{a\}} t\s^{-2\A} \left(\frac{t \s'}{\s}\right)^{-1/2} \sa A\n w, \n w\da G dx - \alpha \int_{B_4\times\{a\}} \s^{-2\A} \left(\frac{t \s'}{\s}\right)^{1/2} w^2 G  dx
\notag
\\
& - N ||b||_{1,1/2} \int_{B_4\times\{a\}}  \s^{-2\A} \left(\frac{t \s'}{\s}\right)^{-1/2} w^2 G dx - N ||V||_{1,1/2} \int_{B_4\times\{a\}} t \s^{-2\A} \left(\frac{t \s'}{\s}\right)^{-1/2} w^2 G dx
\notag
\\
& - N \int |x| \s^{-2\A}\left(\frac{t \s'}{\s}\right)^{-1/2} |\n w|^2  dX -N \int  \frac{|x|^3}{t^2}\s^{-2\A}\left(\frac{t \s'}{\s}\right)^{-1/2} w^2 G dX 
\notag
\\
& - N \A \int \s^{-2\A}\frac{|x|}{t} \left(\frac{t \s'}{\s}\right)^{1/2} w^2 G dX -N\int_{B_4\times\{a\}} \left(|x|+\frac{|x|^3}{t}\right)\s^{-2\A} \left(\frac{t \s'}{\s}\right)^{-1/2} w^2 G dx.
\notag
\end{align}	
It might be helpful for the reader to note that, except for the last four terms in the right-hand side of \eqref{vi25b}, all other terms are similar to those in \eqref{I23bis}.

Next, we analyse $\mi_3$. We apply Proposition \ref{P:vrellich} with the choice $\mathscr Z= \sum_{k=1}^n b_k(x)D_k,$ $B=A$ and $f=\w.$ For this choice of $\mathscr Z,$ it is easy to see that $[D_j, \mathscr Z]f= \sum_{k=1}^nD_jb_k \p_k f.$ Thus, we  have
  \begin{align*}
  |\mi_3|&= \left|2 \int t^{-\frac{n}{2}+1} \left(\frac{t \s'}{\s}\right)^{-1/2}\langle b(x), \n \w\rangle\D(A \n \w) dX \right| = \left|\int t^{-\frac{n}{2}+1} \left(\frac{t \s'}{\s}\right)^{-1/2} \D b \langle A \n \w , \n \w\rangle dX \right|
  \\
  &+2 \int t^{-\frac{n}{2}+1} \left(\frac{t \s'}{\s}\right)^{-1/2} |a_{j\ell} \w_\ell D_jb_k  \w_k| dX + \int t^{-\frac{n}{2}+1} \left(\frac{t \s'}{\s}\right)^{-1/2} |b_k D_k a_{ij}\w_i \w_j| dX
  \\
  & \le N||b||_{1,1/2} \int t^{-\frac{n}{2}+1} \left(\frac{t \s'}{\s}\right)^{-1/2}|\n \w|^2 dX. 
  \end{align*}
We  thus infer for some universal $N$
\begin{align*}
\mi_3 \ge -N ||b||_{1,1/2} \int t^{-\frac{n}{2}+1} \left(\frac{t \s'}{\s}\right)^{-1/2}|\n \w|^2 dX.
\end{align*}
Keeping in mind that \eqref{vtw} gives \eqref{nw}, using \eqref{tricky} we find, similarly to \eqref{I4}, 
\begin{align}\label{vi3c}
  \mi_3 \ge -N ||b||_{1,1/2} \int t \s^{-2\A} \left(\frac{t \s'}{\s}\right)^{-1/2}|\n w|^2 G dX -N ||b||_{1,1/2} \int  \s^{-2\A} \left(\frac{t \s'}{\s}\right)^{-1/2}w^2 G dX.
\end{align}

We now estimate $\mi_4$. It is worth noting that, when $A \equiv \mathbb I_n$, this term vanishes and therefore is not present. Also, notice that since $A$ is only Lipschitz continuous in the $x$-variable, the term $\D(A(x)x)$ cannot be further differentiated, which presently prevents the application of integration by parts. To treat this term, we make use of $\mi_1$, which provides a positive contribution. We start with $\D (A(x)x)=n+O(|x|)$, to find 
\begin{align}\label{vi4}
&|\mi_4|= \left|\frac12 \int t^{-\frac n2} \left(\frac{t \s'}{\s}\right)^{-1/2} [\D(A(x) x)-n]\left( \frac{ \F \w}{2t} + \langle b,\n \w\rangle\right)\w  dX\right|
\\
&\le N \int |x| t^{-\frac{n}2} \left(\frac{t \s'}{\s}\right)^{-1/2} \left| \frac{ \F \w}{2t} + \langle b, \n \w\rangle\right||\w|  dX
\notag
\\
&\le N^2 \int \frac{|x|^2}{t} t^{-\frac{n}2} \left(\frac{t \s'}{\s}\right)^{-1/2} \w^2  dX +\frac 14 \int  t^{-\frac{n}2+1} \left(\frac{t \s'}{\s}\right)^{-1/2}    \left( \frac{ \F \w}{2t} + \langle b, \n \w\rangle\right)^2 dX,
\notag
\end{align}
where we have used the Cauchy-Schwarz inequality. Observe that the last term in the right-hand side of \eqref{vi4} is $\frac 14 \mi_1.$
Using this fact, we deduce that 
\begin{align*}
\mi_1+\mi_4 \ge -N^2 \int \frac{|x|^2}{t} t^{-\frac{n}2} \left(\frac{t \s'}{\s}\right)^{-1/2} \w^2  dX.
\end{align*}
Keeping \eqref{vtw} in mind, along with the fact that $|x|\le 4$ in the domain of the integral, the latter inequality gives
\begin{align}\label{vi14c}
\mi_1+\mi_4 \ge -N^2 \int \frac{|x|}{t} \s^{-2\A} \left(\frac{t \s'}{\s}\right)^{-\frac12} w^2 G dX.
\end{align}

Finally, we analyse $\mi_6.$ Keeping \eqref{beta} in mind and recalling $\I_5$, we notice that 
  \begin{align*}
  \mi_6 &= \I_5 + 2\int t^{-\frac{n}2 +1}\left(\frac{t \s'}{\s}\right)^{-\frac12}\frac{\sa (A-I)x, x \da}{16t^2}\w \langle b, \n \w\rangle dX\\
  & = \I_5 + \frac 18\int t^{-\frac{n}2 -1}\left(\frac{t \s'}{\s}\right)^{-\frac12} \D(\sa (A-I)x, x \da b)\w^2 dX.
  \end{align*}
  
Using \eqref{ass}, $A(0,0) =\mathbb I_n$ and  $x \in B_4,$ one can easily verify that 
\begin{align*}
|\D(\sa (A-I)x, x \da b)| \le N||b||_{1,1/2}|x|.
\end{align*} 
Utilizing this bound, along with \eqref{I5}, yields
\begin{align}\label{vi6}
\mi_6 & \ge - \alpha  ||b||_{1,1/2} \int \s^{-2\A} \left(\frac{t \s'}{\s}\right)^{1/2} w^2 G dX - \frac n4 ||b||_{1,1/2} \int \s^{-2\A} \left(\frac{t \s'}{\s}\right)^{-1/2} w^2 G dX
\\
& - N ||b||_{1,1/2} \int \s^{-2\A} \frac{|x|}t  \left(\frac{t \s'}{\s}\right)^{-1/2} w^2 G dX - N ||b||_{1,1/2}^2 \int \s^{-2\A} \left(\frac{t \s'}{\s}\right)^{-1/2} w^2 G  dX
\notag
\\
& - N ||V||_{1,1/2} ||b||_{1,1/2} \int \s^{-2\A}  t \left(\frac{t \s'}{\s}\right)^{-1/2} w^2 G  dX.
\notag
\end{align}
We can now argue as in the case $A \equiv \mathbb I_n$ (compare with the work leading from \eqref{ode} to \eqref{csf}), to find that for $\A \ge N(1+||b||_{1,1/2}+||V||_{1,1/2}^{1/2}),$ we have
\begin{align}\label{vsf}
&N^{3/2}\int \s^{1-2\A}\left( w_t + \D(A \n w) +\langle b(x), \n w\rangle+ V(x) w\right)^2 G dX
\\
& \ge \int \left(\frac{\theta(\lambda t)}{Nt}  -\frac{\A}{2N}\right)  \s^{1-2\A}  |\n w|^2 G  dX
+ \int \left( \frac{\A}{2N} \frac{\theta(\lambda t)}{t} - 2 \A^2 \right) \s^{-2\A}   w^2 G dX
 \notag
 \\
 & -N \A e^{2\A N}\lambda^{2\A+N}\int w^2dX 
 + \frac{1}{N} \int_{B_4\times\{a\}} t \s^{-2\A}   |\n w|^2 G dx - 2 N \A \int_{B_4\times\{a\}} \s^{-2\A}  w^2 G dx + \mathscr E, \notag
 \end{align}
where $\mathscr E$ is a quantity which includes the last four terms in the right-hand side of \eqref{vi25b}, plus the term in the right-hand side of \eqref{vi14c}. Specifically, 
\begin{align}\label{err}
\mathscr E = &  \int |x| \s^{-2\A}\left(\frac{t \s'}{\s}\right)^{-1/2} |\n w|^2  dX -N \int  \frac{|x|^3}{t^2}\s^{-2\A}\left(\frac{t \s'}{\s}\right)^{-1/2} w^2 G dX 
\\
& -N \A \int \frac{|x|}{t} \left(\frac{t \s'}{\s}\right)^{1/2} w^2 G dX
 -N\int_{B_4\times\{a\}} \left(|x|+\frac{|x|^3}{t}\right)\s^{-2\A} \left(\frac{t \s'}{\s}\right)^{-1/2} w^2 G dx
  \notag
  \\
 &-N^2 \int \frac{|x|}{t} \s^{-2\A} \left(\frac{t \s'}{\s}\right)^{-\frac12} w^2 G dX.
\notag
\end{align}
To estimate $\mathscr E$, we use the bounds $\frac{1}{N} t \leq \s \leq t$ in Lemma \ref{sig} in the first term of \eqref{err}, as well as \eqref{ic}, to eliminate the weights $\left(\frac{t \s'}{\s}\right)^{\pm 1/2}$. We thus find
 \begin{align*}
|\mathscr E|  &\le N  \int \frac{|x|}{t} \s^{1-2\A} |\n w|^2  dX + N\A \int \s^{-2\A} \left(\frac{|x|}{t} +\frac{|x|^3}{\A t^2} \right) w^2 G dX 
\\
& + N \int_{B_4\times\{a\}} \s^{-2\A}  \left(1+\frac{|x|^3}{t}\right)w^2 G dx.
\end{align*}
In the first and second term in the right-hand side of the latter inequality, we use Lemma \ref{logi}, obtaining 
\begin{align}\label{err2}
|\mathscr E| &\le  N e^{2\A N}\lambda^{2\A+N}\int t|\n w|^2dX
+ N \delta \int \s^{1-2\A}\frac{\theta(\lambda t)}{t}|\n w|^2 G dX 
\\
&  + N \A e^{2\A N}\lambda^{2\A+N}\int w^2dX  + N \A \delta \int \s^{-2\A}\frac{\theta(\lambda t)}{t}w^2 G dX
\notag
\\
& + N \int_{B_4\times\{a\}} \s^{-2\A}  w^2 G dx  + N \int_{B_4\times\{a\}} \s^{-2\A}  \frac{|x|^3}{t} w^2 G dx.
\notag
\end{align}
To ensure a  certain smallness of the last term in the right-hand side of \eqref{err2},  we split the integral as follows 
\begin{align*}
\int_{B_4\times\{a\}} \s^{-2\A}  \frac{|x|^3}{t} w^2 G dx =\int_{B_{\delta}\times\{a\}} \s^{-2\A}  \frac{|x|^3}{t} w^2 G dx+\int_{(B_4\setminus B_{\delta})\times\{a\}} \s^{-2\A}  \frac{|x|^3}{t} w^2 G dx.
\end{align*}
Next, we observe that for $|x| > \delta,$ one has $\s^{-2\A}(a) \frac{|x|^3}{a}G(x,a) \le N^{2\A}\lambda^{2\A+N}.$ Therefore, we obtain
\begin{align}\label{bder}
\int_{B_4\times\{a\}} \s^{-2\A}  \frac{|x|^3}{t} w^2 G dx 
&\le \delta \int_{\R^n \times\{a\}} \s^{-2\A}  \frac{|x|^2}{t} w^2 G dx +  N^{2\A}\lambda^{2\A+N} 	\int_{B_4\times\{a\}} w^2 dx 
\\
& \le N \delta 	\int_{B_4\times\{a\}} t \s^{-2\A}  |\n w|^2 G dx + 	N\delta \int_{B_4\times\{a\}} \s^{-2\A}   w^2 G dx
\notag 
\\
& +N^{2\A}\lambda^{2\A+N} 	\int_{B_4\times\{a\}} w^2 dx,
\notag
\end{align}
where in the last inequality we have applied Lemma \ref{logb}.
Combining \eqref{err2} with \eqref{bder}, we thus find
\begin{align}\label{errf}
\mathscr E &\ge  -N \A e^{2\A N}\lambda^{2\A+N} \sup_{t \ge a}\int_{\R^n} [w^2+t|\n w|^2]dx
- N \delta \int \s^{1-2\A}\frac{\theta(\lambda t)}{t}|\n w|^2 G dX 
\\
&  - N \A \delta \int \s^{-2\A}\frac{\theta(\lambda t)}{t}w^2 G dX
- N \int_{B_4\times\{a\}} \s^{-2\A}  w^2 G dx - N \delta 	\int_{B_4\times\{a\}} t \s^{-2\A}  |\n w|^2 G dx.
\notag
\end{align}
Using \eqref{errf} in \eqref{vsf}, we obtain
\begin{align}\label{vsf1}
&N^{3/2}\int \s^{1-2\A}\left( w_t + \D(A \n w) +\langle b, \n w\rangle+ V w\right)^2 G dX
\\
& \ge \int \left[\left(\frac{1}{N}-N\delta\right)\frac{\theta(\lambda t)}{t}  -\frac{\A}{2N}\right]  \s^{1-2\A}  |\n w|^2 G  dX
\notag
\\
&  + \int \left[ \left(\frac{\A}{2N}-N\A \delta\right) \frac{\theta(\lambda t)}{t} - 2 \A^2 \right]\s^{-2\A}   w^2 G dX
\notag
\\
&  - N \A e^{2\A N}\lambda^{2\A+N} \sup_{t \ge a}\int_{\R^n} [w^2+t|\n w|^2]dx
\notag
\\
& \ + \left(\frac{1}{N}-N\delta\right) \int_{B_4\times\{a\}} t \s^{-2\A}   |\n w|^2 G dx - 2 N \A \int_{B_4\times\{a\}} \s^{-2\A}  w^2 G dx. \notag
\end{align}
At this point we choose $\delta >0$ such that 
\begin{align}\label{delv}
\frac 1{2N}-N\delta \ge \frac{1}{4N}. 
\end{align}
Therefore, using \eqref{theta1} and \eqref{delv}  in \eqref{vsf1}, we finally conclude that  for $\A \ge N(1+||b||_{1,1/2}+||V||_{1,1/2}^{1/2}),$ the following inequality holds
\begin{align}\label{vsf2}
&N^{3/2}\int \s^{1-2\A}\left( w_t + \D(A \n w) +\langle b(x), \n w\rangle+ V(x) w\right)^2 G dX
\\
& \ge \frac{\A}{4N}\int  \s^{1-2\A}  |\n w|^2 G  dX
+ \frac{\A^2}{4N} \int \s^{-2\A}   w^2 G dX
- N \A e^{2\A N}\lambda^{2\A+N} \sup_{t \ge a}\int_{\R^n} [w^2+t|\n w|^2]dx
\notag
\\
& \
+\frac{1}{4N} \int_{B_4\times\{a\}} t \s^{-2\A}   |\n w|^2 G dx - 2 N \A \int_{B_4\times\{a\}} \s^{-2\A}  w^2 G dx. \notag
\end{align}
This completes  the proof of the Lemma \ref{carleman} for time-independent coefficient matrix $A$, drift $b$ and potential $V$.

\vskip 0.2in 
			 
\noindent \textbf{Step 3:} $\boxed{\text{In this step we prove \eqref{carl} for a general}\ $A(x,t)$.}$ One remarkable aspect of this step is that the sharp dependence on the norms of the lower-order coefficients in \textbf{Step 2} will cascade into the relevant estimates. While in this step we closely follow some of the ideas in \cite{EFV} and \cite{AB}, it is important for the reader to keep in mind that our ultimate goal is the sharp quantitative dependence in \eqref{K} in Theorem \ref{main} on the norms of the lower-order coefficients $b(x,t)$ and $V(x,t)$, whereas the work \cite{EFV} contains no treatment of them, and \cite{AB} only treats the case of the zero-order term $V(x,t)$. 

With this being said, as in the \textbf{Step 2} we continue to indicate $b(x), V(x)$ by $b, V$, respectively. Moreover, unless otherwise specified, all space-time integrals will be supported in $B_4 \times [0,\frac{1}{4 \lambda}]$ and, as before, we will refrain from explicitly indicating the domain of integration. We  define 
\[
L_0 w :=w_t + \D(A(x,0) \n w)+\langle b, \n w \rangle+Vw,
\]
and observe that 
\[
L_0w = w_t + \D((A(x,0)-A(x,t))\n w)+\D(A(x,t) \n w)+\langle b, \n w \rangle+Vw.
\]
It thus follows
\begin{align*}
& N \int \s_a^{1-2\A}(w_t + \D(A(x,0) \n w) +\langle b, \n w \rangle +Vw)^2G_a dX
\\
& \le 2 N \int \s_a^{1-2\A}(\D((A(x,0)-A(x,t)) \n w)^2 G_adX 
\\
& +2 N \int \s_a^{1-2\A}(w_t + \D(A(x,t) \n w)+\langle b, \n w \rangle+Vw)^2G_adX.
\notag
\end{align*}
This suggests that, in order to replace $A(x,0)$ by $A(x,t)$  in \eqref{carl}, we must control  
\begin{align*}
N \int \s_a^{1-2\A}(\D((A(x,0)-A(x,t)) \n w)^2 G_adX.
\end{align*}
To simplify the analysis, we further notice that  
\[
\D((A(x,0)-A(x,t)) \n w)=\sum_{i,j=1}^n \left\{\partial_i(a_{ij}(x,0)-a_{ij}(x,t))\p_j w+(a_{ij}(x,0)-a_{ij}(x,t))\p_{ij}w\right\}.
\]
Also, we have from \eqref{ass}
\[
|a_{ij}(x,0)-a_{ij}(x,t)| \le M \sqrt{t}\le M \sqrt{t+a}\le MN\s_a(t).
\]
We thus find
\begin{align}\label{dis}
& N \int \s_a^{1-2\A}(\D((A(x,0)-A(x,t)) \n w)^2 G_adX 
\\
& \le 	NM^2 \int \s_a^{1-2\A} |\n w|^2 G_adX + N^2 M^2 \int \s_a^{2-2\A}|D^2w|^2G_adX.
\notag
\end{align}
At this point, we want to estimate each of the two integrals appearing in the right-hand side of \eqref{dis}. Noting that $NM^2 \int \s_a^{1-2\A} |\n w|^2 G_adX$ can be absorbed in the term $\alpha \int \s_a^{1-2\A} |\n w|^2 G_a$  which appears in the left-hand side of \eqref{carl}, we are thus left with estimating the remaining term $\int \s_a^{2-2\A}|D^2w|^2G_adX$  with the right quantitative dependence on $||b||_{1,1/2}$ and $||V||_{1,1/2}$. This is done as follows. We introduce the operator
\[
H  = \D(A(x, 0) \n ) - \p_t.
\]
With $G$ as in \eqref{G}, we note that, using \eqref{nG}, \eqref{Gt} and the assumption \eqref{ass}, along with $A(0,0)=\mathbb I_n$, we have
\begin{equation}\label{HG}
|H G| \le N\left( \frac{|x|}{t}+ \frac{|x|^3}{t^2}\right)G.
\end{equation}
By a computation similar to \cite[(3.97)-(3.103)]{AB} we have
\begin{align*}
&\frac{6}{4N}	\int \s^{2-2\A}|D^2 w|^2 GdX  \le \int \s^{2-2\A} |\n w|^2 |H G|dX +3\A\int \s^{1-2\A}|\n w|^2 GdX- \int_{\{t=a\}} \s^{2-2\A}|\n w|^2 Gdx \notag\\
& + \int \s^{2-2\A} |\n w|^2 \frac{|x|^2}{16t^2}GdX
+\frac{(1+4nN)\delta^2}{\A}\int \s^{1-2\A} (Lw)^2GdX+\frac{\delta^4}{\A^2}\int \s^{-2\A}|\n V|^2w^2GdX\\
&+\frac{\delta^2}{\A}\int \s^{1-2\A}| V||\n w|^2GdX+2n \int \s^{2-2\A}||b||_{1,1/2}|\n w|^2 GdX +2n \int \s^{2-2\A}||b||_{1,1/2}|\n w||D^2w| GdX,
\notag
\end{align*}
The last term in the right-hand side is estimated using the Cauchy-Schwarz inequality in the following way
\[
2n \int \s^{2-2\A}||b||_{1,1/2}|\n w||D^2w| GdX \leq \frac{6}{8N}\int \s^{2-2\A}|D^2w|^2GdX+\frac{2n^2N^2}{3}||b||_{1,1/2}^2\int \s^{2-2\A}|\n w|^2	GdX.	\]				
We thus obtain
\begin{align}\label{a22}
& \frac{6}{8N}\int \s^{2-2\A}|D^2 w|^2 GdX  \le \int \s^{2-2\A} |\n w|^2 |H G|dX 
\\
& +3\A\int \s^{1-2\A}|\n w|^2 GdX- \int_{\{t=a\}} \s^{2-2\A}|\n w|^2 Gdx + \int \s^{2-2\A} |\n w|^2 \frac{|x|^2}{16t^2}GdX
\notag\\
& +\frac{(1+4nN)\delta^2}{\A}\int \s^{1-2\A} (Lw)^2GdX+\frac{\delta^4}{\A^2}\int \s^{-2\A}|\n V|^2w^2GdX
\notag
\\
&+\frac{\delta^2}{\A}\int \s^{1-2\A}| V||\n w|^2GdX+\left(2n||b||_{1,1/2}+\frac{2n^2N^2}{3}||b||_{1,1/2}^2\right) \int \s^{2-2\A}|\n w|^2 GdX.
\notag
\end{align}
We now take $\A$ larger enough such that $\A \ge N (||V||_{1,1/2}^{1/2}+||b||_{1,1/2}).$ Consequently, using in \eqref{a22} the bounds $||V||_{1,1/2} \le \A^2/N^2$, $||b||_{1,1/2} \le \A/N$,
and  $\s \le \frac{1}{4 \lambda}$, we find
\begin{align}\label{a23}
& \frac{6}{8N}\int \s^{2-2\A}|D^2 w|^2 G dX \le \int \s^{2-2\A} |\n w|^2 |H G| dX
\\
& +3\A\int \s^{1-2\A}|\n w|^2 GdX- \int_{\{t=a\}} \s^{2-2\A}|\n w|^2 Gdx  +\int \s^{2-2\A} |\n w|^2 \frac{|x|^2}{t^2}GdX
\notag\\
& +\frac{(1+4nN)\delta^2}{\A}\int \s^{1-2\A} (L_0w)^2GdX
+\frac{\A^2\delta^4}{N^4}\int \s^{-2\A}w^2GdX\notag\\
& +\frac{\A \delta^2}{N^2}\int \s^{1-2\A}|\n w|^2G dX +\left(\frac{n\A}{2N\lambda}+\frac{2n^2\A^2}{12 \lambda}\right) \int \s^{1-2\A}|\n w|^2 GdX.\notag
\end{align}
Since $\s\le t$ and $x \in B_4$, we have from \eqref{HG}
\begin{align}\label{a24}
& \int \s^{2-2\A} |\n w|^2 |H G|dX+ \int \s^{2-2\A} |\n w|^2 \frac{|x|^2}{t^2}GdX
\\
&\le N \int t \s^{1-2\A} |\n w|^2\left( \frac{|x|}{t}+ \frac{|x|^3}{t^2}+\frac{|x|^2}{t^2}\right)GdX
\notag\\
&\le N \int  \s^{1-2\A} |\n w|^2\left( 1+ \frac{|x|^2}{t}\right)GdX\notag\\
&\le \delta NN^{2\A}\lambda^{2\A+N}\int t|\n w|^2dX+ N \delta^2 \int \s^{1-2\A}\frac{\theta(\lambda t)}{t}|\n w|^2 GdX +N \int \s^{1-2\A}|\n w|^2 GdX,\notag
\end{align}
where the last inequality follows from Lemma \ref{logi}. We  now use \eqref{a24} in \eqref{a23} to find for $\alpha$ large enough  
\begin{align}\label{av25}
&\frac{6}{8N}\int \s^{2-2\A}|D^2 w|^2 G dX \le \delta NN^{2\A}\lambda^{2\A+N}\int t|\n w|^2dX
\\
&+ N \delta^2 \int \s^{1-2\A}\frac{\theta(\lambda t)}{t}|\n w|^2 GdX +4\A\int \s^{1-2\A}|\n w|^2 GdX
\notag\\
&- \int_{\{t=a\}} \s^{2-2\A}|\n w|^2 Gdx +\frac{(1+4nN)\delta^2}{\A}\int \s^{1-2\A} (Lw)^2GdX
+\frac{\A^2\delta^4}{N^4}\int \s^{-2\A}w^2GdX
\notag\\
& +\frac{\A \delta^2}{N^2}\int \s^{1-2\A}|\n w|^2GdX +\left(\frac{n\A}{2N\lambda}+\frac{2n^2\A^2}{12 \lambda}\right) \int \s^{1-2\A}|\n w|^2 GdX.\notag
\end{align} 
Now for large enough $\alpha$ satisfying \eqref{quanta}, if as in the statement of Lemma \ref{carleman} we take
\[
\la = \frac{\A}{\delta^2}, \ \ \ \ \ \ \text{and}\ \ \ \ \ \lambda \le \frac{\theta(\lambda t)}{t},
\]
then it follows from \eqref{av25} that
\begin{align}\label{a25}
&\frac{6}{8N}\int \s^{2-2\A}|D^2 w|^2 G dX \le \delta NN^{2\A}\lambda^{2\A+N}\int t|\n w|^2dX
\\
& +\frac{(1+4nN)\delta^2}{\A}\int \s^{1-2\A} (L_0w)^2GdX  - \int_{\{t=a\}} \s^{2-2\A}|\n w|^2 Gdx 
\notag\\
& +\left(N+4+\frac{\delta^2}{N^2}+\frac{n}{2N}+\frac{2n^2\delta^2}{12} \right)\delta^2 \int \s^{1-2\A}\frac{\theta(\lambda t)}{t}|\n w|^2 GdX+\frac{\A^2\delta^4}{N^4}\int \s^{-2\A}w^2GdX.\notag
\end{align}
It now easily follows from \eqref{a25} that, for a possibly larger $N$, the following inequality holds 
\begin{align*}
&\int \s^{2-2\A}|D^2 w|^2 GdX  \le \delta NN^{2\A}\lambda^{2\A+N}\int t|\n w|^2dX+ N\delta^2 \int \s^{1-2\A}\frac{\theta(\lambda t)}{t}|\n w|^2 G dX\\
&- \int_{\{t=a\}} \s^{2-2\A}|\n w|^2 Gdx +N \delta^2\int \s^{1-2\A} (Lw)^2GdX
+\A^2\delta^4\int \s^{-2\A}w^2GdX.
\end{align*}
This estimate generalises (3.110) in \cite{AB}, where the case $b(x)\equiv 0$ in \eqref{meq} was studied. With this result in hands we can now proceed exactly as in their work and reach the following conclusion 
\begin{align}\label{carl15}
&\A^2 \int_{B_4 \times [0,\frac{1}{4\lambda}) }\s_a^{-2\A}w^2G_adX +\A \int_{B_4 \times [0,\frac{1}{4\lambda}) } \s_a^{1-2\A}|\n w|^2G_adX
\\
& \le N \int_{B_4 \times [0,\frac{1}{4\lambda}) } \s_a^{1-2\A}(w_t+\D(A(x,t) \n w)+\langle b(x,0), \n w \rangle+V(x,0)w)^2 G_a dX \notag\\
& +N^{2\A}\A^{2\A}\underset{t \ge 0}{\operatorname{sup}}\int (w^2
+|\n w|^2) dx+\s(a)^{-2\A}\left(-\frac{a}{N}\int |\n w (x,0)|^2G(x,a)dx + N \A \int w^2(x,0)G(x,a)dx\right),
\notag
\end{align}
which proves \eqref{carl} in the hypothesis of \textbf{Step 3}.			
			
\vskip 0.2in			
			
\noindent \textbf{Step 4:} $\boxed{\text{Replacing}\ $b(x,0)$\ \text{by}\ $b(x,t)$,\ \text{and}\ $V(x,0)$\ \text{by}\ $V(x,t)$\ \text{in}\  \eqref{carl15}.}$ This will complete the proof of Theorem \ref{carleman}. By the Cauchy-Schwarz and triangle inequalities, we have
\begin{align*}
&|w_t + (\D(A(x,t) \n w)+V(x,0)w +\langle b(x,0), \n w \rangle)|^2\\&=|w_t+ (\D(A(x,t) \n w)+V(x,t)w +\langle b(x,t), \n w \rangle)+(V(x,0)-V(x,t))w +\langle b(x,t)-b(x,0), \n w \rangle|^2\\
&\le 3(w_t + \D(A(x,t) \n w)+V(x,t)w +\langle b(x,t), \n w \rangle)^2+3(V(x,0)-V(x,t))^2w^2 + 3|b(x,t)-b(x,0)|^2 |\nabla w |^2.
\end{align*}
By \eqref{ass} we have 
\begin{equation*}
(V(x,0)-V(x,t))^2 \le ||V||_{1,1/2}^2\ t,
\ \ \ \ \ 
|b(x,0)-b(x,t)|^2 \le ||b||_{1,1/2}^2\ t.
\end{equation*}
Using this, we find 
\begin{align}\label{v10}
&N \int \s_a^{1-2\A}(w_t + \D(A(x,t) \n w)+w_t+V(x,0)w +\langle b(x,0), \n w \rangle)^2 G_adX\\ &\le  3N \int \s_a^{1-2\A}(w_t+ \D(A(x,t) \n w)+\langle b(x,t), \n w \rangle +V(x,t)w )^2 G_adX \notag\\&+3N||V||_{1,1/2}^2 \int \s_a^{1-2\A}tw^2 G_adX + 3N ||b||_{1,1/2}^2 \int \s_a^{1-2\A}t|\n w|^2 G_a dX.\notag
\end{align}
We would be done if  we could absorb the last two terms in the right-hand side of \eqref{v10}  into  the left-hand side of \eqref{carl15}. Since $\lambda t \le 1/2$  and $\s_a(t) \le t+a \le 1/\lambda$, we have
\begin{align}\label{kv1}
&3 N||V||_{1,1/2}^2 \int \s_a^{1-2\A}tw^2 G_adX+2N ||b||_{1,1/2}^2 \int \s_a^{1-2\A}t|\n w|^2 G_a dX\\ &   \le 	\frac{ 3N}{4\lambda^2}||V||_{1,1/2}^2 \int \s_a^{-2\A}w^2 G_a dX +\frac{3N}{2\lambda} ||b||_{1,1/2}^2 \int \s_a^{1-2\A} |\nabla w|^2 G_a dX.\notag
\end{align}
Inserting \eqref{v10} and \eqref{kv1} in \eqref{carl15}, and keeping in mind that, for $\delta \leq \frac{1}{\sqrt{2}}$, we have $\alpha \leq \frac{\A}{2\delta^2} = \frac{\lambda}{2}$, we obtain
\begin{align}\label{f2}
&	\A^2 \int \s_a^{-2\A}w^2G_a dX +\A \int \s_a^{1-2\A}|\n w|^2G_adX\\
& \le N \int \s_a^{1-2\A}(w_t + \D(A(x,t) \n w)+V(x,t)w)^2 G_adX\notag\\
&+\frac{ N}{\A^2}||V||_{1,1/2}^2 \int \s_a^{-2\A}w^2 G_adX +\frac{N}{\alpha} ||b||_{1,1/2}^2 \int \s_a^{1-2\A} |\nabla w|^2 G_a dX+N^{2\A}\A^{2\A}\underset{t \ge 0}{\operatorname{sup}}\int [w^2
+|\n w|^2] dx\notag\\
&+\s(a)^{-2\A}\left(-\frac{a}{N}\int |\n w (x,0)|^2G(x,a)dx + N \A \int w^2(x,0)G(x,a)dx\right).\notag
\end{align}
Now observe that the second and third term in the right-hand side of \eqref{f2} can be absorbed in the first and second term of left-hand side of \eqref{f2}, provided that 
\begin{equation*}
\frac{\A^2}{2} \ge \frac{ N}{\alpha^2}||V||_{1,1/2}^2
\ \ \ \text{and}\ \ \ \ 
\frac{\alpha}{2} \geq \frac{N}{\alpha}||b||_{1,1/2}^2.
\end{equation*}
This is ensured if $\A \ge 4N( ||V||_{1,1/2}^{1/2}+||b||_{1,1/2})$. If for $N$ large we thus choose $\A \ge N(||V||_1^{1/2}+ ||b||_{1,1/2}+1)$, the conclusion follows.			
			
\end{proof}


\section{Proof of Theorem \ref{main}}\label{s:main}

In order to establish the quantitative uniqueness result in Theorem \ref{main}, we also need the following monotonicity in time result. This is analogous to \cite[Lemma 1]{EFV}, except that in our situation we need an inequality with a precise quantitative dependence on the norms of the drift and the potential. As a help to the reader, we preliminarily set some relevant notation, and mention some critical aspects of what will unfold. Let $u$ be a solution to \eqref{meq} in $Q_4$ such that $u(\cdot, 0) \not \equiv  0$ in $B_1$. For every $\rho \in (0,1]$, we set 
\begin{equation}\label{Thetarho}
\Theta_{\rho}=\frac{\int_{Q_{3}} u^2( x, t)dX}{\rho^2 \int_{B_{\rho}}u^2( x,0)dx},
\end{equation}
and for any $N>1$ also let
\begin{equation}\label{TrhoN}
T_{\rho,N} = \frac{\rho^2}{2N\operatorname{log}(2N(1+||V||_{\infty}+||b||_{\infty}^2)^2{\Theta_{\rho}})+ 5 N^2 (||V||_{\infty}^{1/2}+||b||_{\infty}+1)}.
\end{equation}
The reader should note that, when $\rho =1$, the quantity $\Theta_1$ in \eqref{Thetarho} is different from $\Theta$ in \eqref{T}. We emphasise that the inequality in \eqref{mon1} below will be proved for a general parameter $\rho \in (0,1]$. 
This is important because, finally, the Carleman estimate in \eqref{carl} above will be used, in combination with \eqref{mon1}, to establish  \eqref{gausineq} below. It is this latter inequality which ultimately leads to the desired vanishing order estimate claimed in Theorem \ref{main}. This is precisely  where a choice of the universal parameter $\rho$ is required. 

\begin{lemma}\label{mon}
Let $u$ be a solution to \eqref{meq} in $Q_4$. Then there exists a universal constant $N>1$ such that, for every $\rho\in (0,1]$ and 
$t\le T_{\rho,N}$, the following inequality holds
\begin{align}\label{mon1}
Ne^{ (||V||_{\infty}^{1/2}+||b||_{\infty})}\int_{B_{2\rho}} u^2(x,t)dx \ge \int_{B_{\rho}} u^2(x,0)dx.
\end{align}
\end{lemma}

\begin{proof}
In view of rescaling, it is enough to prove \eqref{mon1} for $\rho=1$. In what follows, without loss of generality, we assume that $A$ be defined on the whole of $\R^n \times \R$, and satisfies  the bounds in \eqref{ass}. For $t>0$  consider the semigroup $P_t = e^{t E_0}$ associated with $E_0 f = -  \operatorname{div}(A\n f)$. It is well-known that 
\begin{itemize}
\item[(i)] $P_t 1 = 1$;
\item[(ii)] $P_t f\to f$, as $t\to 0^+$,
\end{itemize}
see \cite[Chapter 1]{Fr}.
Let now $\phi \in C^{\infty}_0(B_2)$ be such that $0 \le \phi \le 1$ in $B_2$ and $\phi = 1$ in $B_{3/2}$, and set $w(x,t)= u(x,t) \phi(x)$. For a given $y\in \Rn$, we introduce the function 
\[
H(t) = P_t(w^2)(y) = \int_{\Rn} w(x,t)^2 \mathcal G(x,y,t) dx,
\] 
where $\mathcal G(x,y,t)$ denotes the fundamental solution, with pole at $(y,0)$ of the operator $\p_t + E_0$.
Clearly, for every $y\in B_1$ we have from (i) and (ii)
\begin{equation}\label{Hzero}
H(t) \underset{t\to 0^+}{\longrightarrow} w(y,0)^2 = u(y,0)^2.
\end{equation}
With $L_0 = \p_t - E_0$, we easily find
\begin{align}\label{pth}
H'(t) &= 2\int_{\R^n} w\ L_0 w\ \mathcal G dx + 2\int_{\R^n} \langle A \n w, \n w \rangle \mathcal G dx.
\end{align}
Since $u$ solves \eqref{meq}, a computation gives
\begin{align*}
L_0 w & = \phi L_0 u + u L_0\phi + 2\sa A\n u,\n \phi\da 
\\
& = - \phi \sa b,\n u\da - V w + u\D(A \n \phi ) + 2\sa A\n u,\n \phi\da 
\\
&=- \langle b, \n w \rangle - Vw +\langle b, \n \phi \rangle u +   u\D(A \n \phi ) +  2 \langle A \n u, \n \phi \rangle.
\end{align*}
Replacing this identity in \eqref{pth} and using the support property of $\phi$, we find 
\begin{align}\label{pth1}
H'(t) \ge   &- 2 ||V||_{\infty} H(t) -2\int_{\R^n}w \langle b, \n w \rangle \mathcal G dx +2 \int_{\R^n} \langle A \n w, \n w \rangle \mathcal G dx
\\
&-  N(1+||b||_{\infty})\int_{B_2\setminus B_{3/2}} |w|(|u|+|\n u|)\mathcal Gdx.
\notag
\end{align}
With $\Lambda$ as in \eqref{ell}, the numerical inequality $|2\alpha\beta|\le \ve \alpha^2 + \ve^{-1} \beta^2$, with $\ve = \frac1{\sqrt{2 \Lambda}}$, yields 
\[
 |w \langle b, \n w \rangle | \le \Lambda ||b||_{\infty}^2 w^2   + \frac1{4 \Lambda} |\n w|^2.
 \] 
Using this estimate, along with $\langle A \n w, \n w \rangle \ge \Lambda^{-1} |\n w|^2$, in \eqref{pth1}, we find
\begin{align*}
H'(t) \ge   &- N(||V||_{\infty}+||b||^2_{\infty}) H(t) -  N(1+||b||_{\infty})\int_{B_2\setminus B_{3/2}} |w|(|u|+|\n u|)\mathcal Gdx.
\end{align*}
By the Gaussian bounds for $\mathcal{G}$ in \cite{Fr}, we thus obtain
\begin{align*}
H'(t) \ge - N(||V||_{\infty}+||b||^2_{\infty}) H(t) - N (1+||b||_{\infty}) t^{- \frac n2} e^{-\frac{1}{4Nt}}\int_{B_2\setminus B_{3/2}} (|u|^2+|\n u|^2) dx.
\end{align*}
By Lemma \ref{len}, we now deduce from the latter inequality (for a new $N$ large) 
\begin{align*}
H'(t) \ge - N(||V||_{\infty}+||b||^2_{\infty}) H(t)  - N (1+||V||_{\infty}+||b||^2_{\infty})^2 t^{- \frac n2}e^{-\frac{1}{4Nt}}\int_{Q_3} u^2dX.
\end{align*}
Since for sufficiently small $t$ one has  $t^{n/2}e^{-\frac{1}{4Nt}} \le  e^{-\frac{1}{8Nt}}$, we obtain 
\begin{align*}
H'(t) \ge - N(||V||_{\infty}+||b||^2_{\infty}) H(t)  - N (1+||V||_{\infty}+||b||^2_{\infty})^2 e^{-\frac{1}{Nt}}\int_{Q_3} u^2dX.
\end{align*}
Integrating this inequality with respect to $t$, and using \eqref{Hzero}, we find
\begin{align*}
&e^{N(||V||_{\infty}+||b||^2_{\infty}) t} \int w^2(x,t) \mathcal G(x,t;y,0) dx\\
& \ge u^2(y,0) - e^{N(||V||_{\infty}+||b||^2_{\infty}) t} (1+||V||_{\infty}+||b||^2_{\infty})^2Nte^{-\frac{1}{Nt}}\int_{Q_3} u^2dX.
\end{align*}
We now integrate with respect to the variable $y\in B_1$, change the order of integration in the left-hand side, and finally  use (i) above  to obtain 
\begin{align}\label{int2}
&e^{N(||V||_{\infty}+||b||^2_{\infty}) t}\int_{B_2} u^2(x,t) dx \ge e^{N(||V||_{\infty}+||b||^2_{\infty}) t} \int_{\R^n} w^2(x,t) dx\\
& \ge \int_{B_1}u^2(x,0) dx -e^{N (||V||_{\infty}+||b||^2_{\infty}) t}  (1+||V||_{\infty}+||b||^2_{\infty})^2Nte^{-\frac{1}{Nt}}\int_{Q_3} u^2dX.\notag
\end{align}
Let $D$ be the universal constant in Lemma \ref{len}. We now choose $N>1$ universal (in particular,  independent of $||V||_{\infty}$ and $||b||_{\infty}$) such that $2N/D >1$ and $ N \log(2N/D)>1.$
Notice that, from \eqref{eq:len},  we find 
\[
1 \le D(1+||V||_{\infty}+||b||_{\infty}^2)\frac{\int_{Q_3} u^2(x,t)dX}{\int_{B_1}u^2(x,0)dx}=D(1+||V||_{\infty}+||b||_{\infty}^2)\Theta_1.
\]  
Thus we have $N\log (2N(1+||V||_{\infty}+||b||_{\infty}^2)^2\Theta_1)= N\log (2N/D)+N\log (D(1+||V||_{\infty}+||b||_{\infty}^2)^2\Theta_1)>1.$  We now rewrite  $e^{N(||V||_{\infty}+||b||_{\infty}^2) t} (1+||V||_{\infty}+||b||_{\infty}^2)^2Ne^{-\frac{1}{Nt}}\int_{Q_3} u^2(x,t)dX$	as	\begin{align}\label{mon2}
&e^{N(||V||_{\infty}+||b||_{\infty}^2) t} (1+||V||_{\infty}+||b||_{\infty}^2)^2Ne^{-\frac{1}{Nt}}\int_{Q_3} u^2(x,t)dX \\
&= e^{N(||V||_{\infty}+||b||_{\infty}^2) t} e^{-\frac{1}{2Nt} } (1+||V||_{\infty}+||b||_{\infty}^2)^2Ne^{-\frac{1}{2Nt}}\int_{Q_3} u^2(x,t)dX.\notag
\end{align}
Now let \begin{equation}\label{small1} t\leq  \left(2N\log(2N(1+||V||_{\infty}+||b||_{\infty}^2)^2\Theta_1) + 5 N^2 ([||V||_{\infty}+||b||_{\infty}^2]^{1/2}+1)\right)^{-1} = T_{1,N},\end{equation}
see the definition \eqref{TrhoN}.
 \eqref{small1} in particular implies that $t \leq \left(5 N^2 ([||V||_{\infty}+||b||_{\infty}^2]^{1/2}+1)\right)^{-1}$, using which we find
\begin{equation}\label{intert1}
e^{N(||V||_{\infty}+||b||_{\infty}^2) t} e^{-\frac{1}{2Nt} }\leq e^{\frac{1}{5N} (||V||_{\infty} + ||b||_{\infty}^2)^{1/2}} e^{-\frac{5N}{2} (||V||_{\infty} + ||b||_{\infty}^2)^{1/2}} \leq 1.
\end{equation}	
Also  \eqref{small1} implies that 
\[
t \leq \frac{1}{2N\log(2N(1+||V||_{\infty}+||b||_{\infty}^2)^2\Theta_1)},
 \]
 which in turn gives
\begin{equation}\label{intert2}
(1+||V||_{\infty}+||b||_{\infty}^2)^2Ne^{-\frac{1}{2Nt}}\int_{Q_3} u^2(x,t)dX  \le	\frac{1}{2} \int_{B_1}u^2(x,0)dx.
\end{equation}
Using \eqref{intert1} and \eqref{intert2} in \eqref{int2}, we thus find
\begin{align}\label{for1}
&e^{N(||V||_{\infty}+||b||^2_{\infty}) t}\int_{B_2} u^2(x,t) dx \le \frac{1}{2} \int_{B_1}u^2(x,0)dx.
\end{align}
Also, similarly to \eqref{intert1}, if we use \eqref{small1} we obtain
\begin{equation}\label{for2}
e^{N(||V||_{\infty}+||b||^2_{\infty}) t} \le 	e^{\frac{1}{5N} (||V||_{\infty} + ||b||_{\infty}^2)^{1/2}}.
\end{equation}
Exploiting \eqref{for2} in \eqref{for1}, we finally obtain 
\begin{align*}
N e^{  ||V||_{\infty}^{1/2}+||b||_{\infty}}\int_{B_2} u^2(x,t)dx \ge \int_{B_1}u^2(x,0)dx,
\end{align*}
for 
\[
0< t \le \left(2N\log[2N(1+||V||_{\infty}+||b||_{\infty}^2)^2\Theta_1] + 5 N^2 ([||V||_{\infty}+||b||_{\infty}^2]^{1/2}+1)\right)^{-1},
\]
where $\Theta_1=\frac{\int_{Q_3} u^2dX}{\int_{B_1}u^2(x,0)dx}.$ This completes the proof of the lemma.	

\end{proof}

We are finally ready to present the 
\begin{proof}[Proof of Theorem \ref{main}] 
Let $u$ be as in the Theorem \ref{main}. We pick 
$\psi \in C^{\infty}_{0}(B_4)$, such that $\psi \equiv 1$ in $B_1$,  $\psi =0$ outside $B_{2}$. Next,  for large $\alpha$, to be later determined, we let $\lambda=\frac{\alpha}{\delta^2}$, and pick $\phi \equiv 1$ in $0\le t \le 1/(8\lambda)$ and $\phi =0$ for $t \ge {1}/({4\lambda})$. With this choices, we consider the function $w(x,t)=u(x,t)\psi(x) \phi(t)$. We obtain from \eqref{meq}
$$w_t + \D(A\n w)+ \langle b, \n w \rangle + Vw = \D(A \n \psi)u \phi + 2 \langle A \n \psi, \n u \rangle \phi  + u \psi \phi_t +\langle b, \n \psi \rangle u\phi.$$
From the support properties of $\psi$ and $\phi$, it thus follows 	
\begin{align*}
\big(w_t + \D(A\n w)+\langle b, \n w \rangle +Vw\big)^2 \le C (1+||b||^2_{\infty})(u^2 +|\n u|^2)\mathbf 1_{B_2\setminus B_{1}}(x) +C\lambda^2u^2 \mathbf 1_{(0,1/4\lambda)\setminus (0,1/8\lambda)}(t), 
\end{align*}
where $\mathbf 1_E$ denotes the indicator function of the set $E$.
With the Carleman estimate \eqref{carl} in Theorem \ref{carleman} and with Lemma \ref{mon} in hand, we can now repeat the arguments in \cite{EFV} (see also  \cite[Proposition 3.7]{AB}) to deduce that, for some $\rho \in (0,1)$ universal (depending on $n$, and the constants $\Lambda$ in \eqref{ell}, and $M$ in \eqref{ass}), the following inequality holds
\begin{align}\label{gausineq}
2a\int |\n w (x,0)|^2G(x,a)dx + \frac{n}{2} \int w^2(x,0)G(x,a)dx\le N^3\A_0 \int w^2(x,0)G(x,a)dx
\end{align}
with 
\[
\A_0=  N \log(2N(1+||V||_{\infty}+||b||_{\infty}^2)^2\Theta_{\rho})+N(||V||_{1,1/2}^{1/2}+||b||_{1,1/2}+1),
\]
where $\Theta_{\rho}$ is as in \eqref{Thetarho}.
We now use Lemma \ref{do} to obtain for all $r<1/2$
\begin{align}\label{doi}
\int_{B_{2r}} u^2(x,0)dx \le Q \int_{B_r} u^2(x,0)dx,
\end{align}
where 
\[
Q=\exp(N^4 \operatorname{log}(2N(1+||V||_{\infty})\Theta_{\rho})+N^4(||V||_1^{1/2}+||b||_{1/2}^{1/2}+1)).
\]
In order to now replace $\Theta_{\rho}$ in \eqref{doi} by the universal normalising quantity $\Theta$ defined by \eqref{T}, we use a covering argument as in \cite{EFV}, or \cite[Theorem 1]{AB}, finally obtaining the following space-like doubling inequality
\begin{align}\label{dub1}
\int_{B_{2r}} u^2(x,0)dx \le \tilde Q \int_{B_r} u^2(x,0)dx,
\end{align}  
where 
\[
\tilde Q= \exp({N \operatorname{log}(N\Theta)+N(||V||_{1,1/2}^{1/2}+||b||_{1,1/2}+1)}),
\]
and $\Theta$ is as in \eqref{T}. At this point, the desired estimate \eqref{df} follows from \eqref{dub1} in the standard way (see, for instance,  (3.140)-(3.146) in \cite{AB}). This completes the proof of the theorem.
		
\end{proof}	
		
In closing, we mention a  quantitative two-ball/one-cylinder inequality, Proposition \ref{2sphere1} below, that will be needed in the proof of the Landis type decay in Section \ref{landis}. Its proof follows from the doubling inequality \eqref{dub1}, and since it is identical to that of \cite[Proposition 3.8]{AB}, we refer the reader to that source. 		
\begin{prop}\label{2sphere1}
Let $u$ be as in Theorem \ref{main}. Then, with 
\[
M_1=N\operatorname{log}(2N(1+||V||_{\infty} +||b||_{\infty}^2))+N(||V||_{1,1/2}^{1/2}+||b||_{1,1/2}+1),
\]
for every $0<r\le 1$ we have 
\[
\int_{B_{2}}u^2(x,0)dx \le e^{\frac{ M_1 \operatorname{log}_2(2/r)}{1+N\operatorname{log}_2(2/r)}} \left(\int_{B_r}u^2(x,0)dx\right)^{\frac{1}{1+N\operatorname{log}_2(2/r)}}\left(N\int_{Q_{4}} u^2( x, t)dX\right)^{\frac{N\operatorname{log}_2(2/r)}{1+N\operatorname{log}_2(2/r)}}. 
\]
\end{prop}	
	

\vskip 0.2in

\section{Proof of Theorem \ref{maindf}}\label{dof}
	
In this section we prove Theorem \ref{maindf}. We will need the following consequence of the analyticity estimates in \cite{Ei}.	
	
\begin{thrm}\label{anal}
Let $A(x,t)$ satisfy  \eqref{ell}, \eqref{ass}, and assume in addition that $x\to A(x, t)$ be real-analytic, uniformly with respect to $t$ in $\Rn \times \R$. Let $u$ solve \eqref{real1} in $Q_4$. Then there exists a universal $r_0 \in (0, 1/4)$ such that $u(\cdot, 0)$ can be extended holomorphically  to some $w$ in the complex ball $B^{\mathbb C^n}_{r_0}$, and moreover the following estimate holds:
\begin{equation}\label{estan1}
\sup_{z \in B^{\mathbb C^n}_{r_0}} |w(z)| \leq e^{C(\sqrt \la +1)} \left(\int_{Q_{4}} u^2 dX \right)^{1/2}.\end{equation}
\end{thrm}

\begin{proof}
For $y \in \R$ we consider the function $v(x,y,t)=e^{\sqrt \la\ y}u(x,t)$. One easily verifies that $v$ solves in $B_4 \times \{|y| < 4\} \times [0, 16)$
\[
v_t + \D( \tilde A(x,y,t) \n v)  =0,
\]
where $\tilde A= (\tilde a_{ij})$ is the following $(n+1) \times (n+1)$ uniformly elliptic,  real-analytic, matrix-valued function 
\[
\begin{cases}
\tilde a_{ij}= a_{ij},\ \ \ \ \ \ \ \ \ \ \ \ \ \ \ \ \ \ \ i, j =1, 2,...,n,
\\
\tilde a_{(n+1)i}=\tilde a_{i(n+1)}=0,\ \ \  i=1, 2,...,n,
\\
\tilde a_{(n+1)(n+1)}= 1.
\end{cases}
\]
From the estimates in \cite[Theorem 6.2, p. 221]{Ei} it follows that there exists $r_0>0$ such that the power series expansion of $v(\cdot,0)$ at $0$ converges absolutely. Therefore, a complex extension $\tilde w$  of $v(\cdot, 0)$ is admissible in $B^{\mathbb C^{n+1}}_{r_0}$, and the following bound holds 
\[
\sup_{z \in B^{\mathbb C^{n+1}}_{r_0}} |\tilde w(z)|  \leq C\left( \int_{B_4 \times \{|y| < 4\} \times [0, 16)} v^2 \right)^{1/2} \leq  e^{C(\sqrt \la +1)} \left(\int_{Q_{4}} u^2 dX \right)^{1/2}.
\]
Going back to $u$, we obtain \eqref{estan1} from the latter estimate.

\end{proof}

We also need the following result from complex analysis which follows by a covering argument from \cite[Lemma 2.5]{Han}.

\begin{thrm}\label{estan2}
Suppose $f: B_1 \subset \mathbb C \to \mathbb C$  is holomorphic, with 
\[
|f(0)|=1,\ \ \ \ \ \ \ \sup_{B_1} |f| \leq 2^N.
\]
Then there exists a universal $C_1(N)>0$ such that 
\[
\#\{z \in B_{1/2}: f(z)=0\} \leq C_1N.
\]
\end{thrm}

We are ready to present the
 	
\begin{proof}[Proof of Theorem \ref{maindf}]
Let $\mathcal{K}$ be as in \eqref{K}, and denote by $w$ the complex extension of $u(\cdot, 0)$ in $B^{\mathbb C^n}_{r_0}$ for which \eqref{estan1} in Theorem \ref{anal} holds. From the vanishing order estimate  \eqref{df} in Theorem \ref{main}, it follows that for some $C=C(r_0)>0$ one has
\[
\int_{B_{r_0/16}} u(x, 0)^2 \geq e^{-C \mathcal{K}}.
\]
With $|u(x_0, 0)| = \underset{B_{r_0/16}} \max |u(\cdot, 0)|$, we infer that, for a different $C>0$,  
\begin{equation}\label{df0}
|u(x_0, 0)| \geq  e^{-C\mathcal{K}}.
\end{equation}
For $\theta \in \mathbb S^{n-1}$ and $z\in \mathbb C$, consider the function
\[
w_{\theta}(z)= \frac{w(x_0 +\theta z)}{u(x_0, 0)}. 
\]
It is clear that $|w_{\theta}(0)|=1$. Since $w$ is defined in $B^{\mathbb C^n}_{r_0}$ and $|x_0| \leq r_0/16$, it is clear that $w_{\theta}$ is well-defined in $\{z\in \mathbb C\mid |z| \leq r_0/2\}$. Moreover from \eqref{estan1} and the definition of $\mathcal{K}$ and $\Theta$ it follows that for $z \in B^{\mathbb C^n}_{r_0}$ one has
\[
|w(z)| \leq e^{C \mathcal{K}} \Theta^{1/2} \left(\int_{B_1} u(x, 0)^2 dx\right)^{1/2} \leq C_0 e^{C_1 \mathcal{K}}, 
\]
where $C_0=\max\{(\int_{B_1} u(x, 0)^2 dx)^{1/2},1\}$.  Combined with \eqref{df0}, this estimate implies that
\begin{equation}\label{df10}
\sup_{B_{r_0/2} \subset \mathbb C} |w_{\theta}(z)| \leq e^{C_2 \mathcal{K}}. 
\end{equation}
This shows that $w_{\theta}(z)$ verifies  the assumptions of Theorem \ref{estan2}, and from this we infer that (for a yet different $C>0$)
\[
\#\{z \in B_{r_0/4}: w_{\theta}(z) =0\} \leq C \mathcal{K}.
\]
Going back to $u(\cdot, 0)$, we deduce from this estimate
\begin{equation}\label{co5}
\#\{r: |r| \leq r_0/4, u( x_0+ r \theta, 0) =0\} \leq C \mathcal{K}.
\end{equation}		
By \eqref{co5} and an application of the coarea formula, we thus conclude
\[
H^{n-1}\{ B_{r_0/4} (x_0) \cap  \{u(\cdot, 0) =0\}\} \leq \int_{\mathbb S^{n-1}} C \mathcal{K} d\theta = C \mathcal{K}.
\]
Since $|x_0| \leq \frac{r_0}{16}$, we obviously  $B_{r_0/8 } \subset B_{r_0/4}(x_0)$, and therefore the latter inequality implies
\begin{equation}\label{co8}
H^{n-1}\{ B_{r_0/8 } \cap  \{u(\cdot, 0) =0\}\} \leq C \mathcal{K}.	\end{equation}			 		 			
Since $r_0$ is universal, the desired conclusion \eqref{df1} now follows from \eqref{co8} by a covering argument.
			
\end{proof}

\medskip
	
\section{Proofs of Theorems \ref{main2}  and \ref{main1}}\label{landis}	

In this last section we prove Theorems \ref{main2}  and \ref{main1}.
	
\begin{proof}[Proof of Theorem \ref{main2}]
Given $x_0 \in \Rn$ with $|x_0| =R$. If $R>>1$, by the rescaling 
\[
w(x,t) = u(Rx +x_0, R^2 t),
\]
we find that $w$ solves in $Q_4$
\[
w_t + \Delta w = R^2 c_0 w+ R \langle b_0, \n w \rangle.
\]
With $r=\frac{1}{R}\le 1$, we apply Proposition \ref{2sphere1}, keeping in mind that now $A=\mathbb I_n$, $b \equiv -R b_0$ and $V=- R^2 c_0$, so that
\[
||V||_{1, 1/2}^{1/2} = |c_0|^{1/2}R,\ \ \ \ \ \ \ 	||b||_{1,1/2}= |b_0| R. 
\]
We thus obtain, with some $C_1=C_1(|c_0|, |b_0|, n)>0$,
\begin{align}\label{3sq}
\int_{B_{2}}w^2(x,0)dx \le e^{C_1R\frac{\operatorname{log}_2 R}{1+N\operatorname{log}_2 R}} \left(\int_{B_r}w^2(x,0)dx\right)^{\frac{1}{1+N\operatorname{log}_2 R}}\left(N\int_{Q_{4}} w^2( x, t)dX\right)^{\frac{N\operatorname{log}_2 R}{1+N\operatorname{log}_2 R}}.\end{align}	
We now use the hypothesis \eqref{expb1} to estimate 
\[
\left(N\int_{Q_{4}} w^2( x, t)dX\right)^{\frac{N\operatorname{log}_2 R}{1+N\operatorname{log}_2 R}} \leq e^{C_2R},
\]
for some $C_2$ depending on the bound $C$ in \eqref{expb1}.
We thus obtain from \eqref{3sq} 
\[
\int_{B_{2}}w^2(x,0)dx \le e^{C_1R\frac{\operatorname{log}_2 R}{1+N\operatorname{log}_2 R}} \left(\int_{B_r}w^2(x,0)dx\right)^{\frac{1}{1+N\operatorname{log}_2 R}} e^{C_2 R}.
\]
Scaling back to $u$, and then raising the resulting inequality to the exponent $1+ N \log_2 R$, we obtain for  some  $N_1$ depending on $n, C_2, |b_0|, |c_0|$,  
\[
\int_{B_{1}(x_0)} u^2(x,0) dx \geq 	e^{-N_1R \log R} \left(\int_{B_{2R}(x_0)}  u^2(x, 0) dx\right)^{N_1\log R}.
\]
Noting that $B_1 \subset B_{2R}(x_0 )$, this inequality immediately implies 
\[
\int_{B_1(x_0)} u^2(x,0) dx \geq e^{-N_1 R \log R} \left(\int_{B_1}  u^2(x, 0) dx\right)^{N_1\log R} = e^{-\tilde CR \log R},
\]
where $\tilde C= \max \big\{N_1[1 - \log (\int_{B_1} u(x,0)^2 dx)], N_1\big\}$.	 This finishes the proof.
		
\end{proof}
		
\medskip
		
Our final objective is proving Theorem \ref{main1}. For this we need a local quantitative analogue of Proposition \ref{2sphere1} which is, in turn, a consequence of the following Lemma \ref{carleman1}. Such result  essentially follows from the work \cite{EFV}, but for completeness we outline the relevant details. In its statement, with $\s$ as in Lemma \ref{ode}, similarly to the statement of Theorem \ref{carleman} above, we indicate $G_a(x,t)=G(x,t+a)$ and $\s_a(t)=\s(t+a)$.
	
\begin{lemma}\label{carleman1}
Let $V\in L^\infty(\Rn\times [0,\infty),\mathbb C)$. There exist universal constants $N$ and $\delta \in (0,1)$ such that for $\A \ge N(1+||V||_{\infty}^{2/3})$, and
with $\lambda=\A/\delta^2,$ the following inequality holds for all $w \in C_{0}^{\infty}(B_4 \times [0,\frac{1}{4\lambda}))$ and $0<a\le \frac{1}{4\lambda}$	\begin{align}\label{carlv}
& \A^2 \int_{B_4 \times [0,\frac{1}{4\lambda}) }\s_a^{-2\A}w^2G_adX +\A \int_{B_4 \times [0,\frac{1}{4\lambda}) } \s_a^{1-2\A}|\n w|^2G_adX\\
& \le N \int_{B_4 \times [0,\frac{1}{4\lambda}) } \s_a^{1-2\A}(w_t+ \Delta w + V(x,t)w)^2 G_adX \notag\\
&+\s(a)^{-2\A}\left(-\frac{a}{N}\int |\n w (x,0)|^2G(x,a)dx + N \A \int w^2(x,0)G(x,a)dx\right).
\notag
\end{align}  
\end{lemma}

\begin{proof}
The following Carleman estimate was proved in \cite{EFV}
\begin{align*}
& \A^2 \int_{B_4 \times [0,\frac{1}{4\lambda}) }\s_a^{-2\A}w^2G_adX +\A \int_{B_4 \times [0,\frac{1}{4\lambda}) } \s_a^{1-2\A}|\n w|^2G_adX
\\
& \le N \int_{B_4 \times [0,\frac{1}{4\lambda}) } \s_a^{1-2\A}(w_t + \Delta w)^2 G_adX 
\notag\\
&+\s(a)^{-2\A}\left(-\frac{a}{N}\int |\n w (x,0)|^2G(x,a)dx + N \A \int w^2(x,0)G(x,a)dx\right).
\notag
\end{align*}
If in this inequality we use  
\[
2(w_t + \Delta w +Vw)^2 \geq (w_t + \Delta w)^2 - 2 ||V||_{\infty}^{2} w^2,
\]
we find
\begin{align}\label{carl2}
&	\A^2 \int_{B_4 \times [0,\frac{1}{4\lambda}) }\s_a^{-2\A}w^2G_adX +\A \int_{B_4 \times [0,\frac{1}{4\lambda}) } \s_a^{1-2\A}|\n w|^2G_adX\\
& \le 2N \int_{B_4 \times [0,\frac{1}{4\lambda}) } \s_a^{1-2\A}( \Delta w + w_t +Vw)^2 G_adX +2 N \int_{B_4 \times [0,\frac{1}{4\lambda}) } \s_a^{1-2\A} ||V||_{\infty}^2 w^2 G_a dX \notag\\
&+\s(a)^{-2\A}\left(-\frac{a}{N}\int |\n w (x,0)|^2G(x,a)dx + N \A \int w^2(x,0)G(x,a)dx\right).\notag
\end{align}
Now note that the assumption $\alpha \geq 5 N^{1/3} ||V||_{\infty}^{2/3}$, implies
\[
2N \s_a^{1-2\alpha} ||V||_{\infty}^2 \leq  \frac{2N}{\alpha} \s_a^{-2\A} \frac{\alpha^3}{125 N}= \frac{2\alpha^2}{125} \s_a^{-2\A}.
\]
In this inequality we have also used that $\s_a(t) \leq t+a \leq \frac{1}{\alpha}$.
We thus infer that, when $\alpha \geq 5N^{1/2} ||V||_{\infty}^{2/3}$, the term $$2N \int_{B_4 \times [0,\frac{1}{4\lambda}) } \s_a^{1-2\A} ||V||_{\infty}^2 w^2 G_a dX$$ in \eqref{carl2} can be absorbed by the term
$$\A^2 \int_{B_4 \times [0,\frac{1}{4\lambda}) }\s_a^{-2\A}w^2G_adX$$ 
in the left hand side of \eqref{carl2}. With a new, larger $N$, the desired conclusion \eqref{carlv} thus follows.

\end{proof}

Combining Lemma \ref{carleman1} with Lemma \ref{mon} (also noting that $||V||_{\infty}^{1/2} \leq ||V||_{\infty}^{2/3}+1$), we can repeat the arguments in \cite{AB} to conclude that the following quantitative two-ball/one-cylinder inequality holds for $L^{\infty}$ potentials.

\begin{prop}\label{2sphere2}
Let $\tilde V \in L^{\infty}( \mathbb R^n \times [0, \infty),\mathbb C)$, and assume that $w$ be a solution in $Q_4$ of the equation
\begin{equation}\label{er1}
w_t + \Delta w  = \tilde V w, 
\end{equation}
with $w(\cdot,  0) \not\equiv 0$. There exists a universal $N >0$ such that, with $M_2=N (||\tilde V||_{\infty}^{2/3} +1)$, one has
\begin{align*}
\int_{B_{1}}w^2(x,0)dx \le e^{M_2\frac{\operatorname{log}_2(2/r)}{1+N\operatorname{log}_2(2/r)}} \left(\int_{B_r}w^2(x,0)dx\right)^{\frac{1}{1+N\operatorname{log}_2(2/r)}}\left(N\int_{Q_{4}} w^2( x, t)dX\right)^{\frac{N\operatorname{log}_2(2/r)}{1+N\operatorname{log}_2(2/r)}}.
\end{align*}		
\end{prop}		
	
We can finally give the
	
\begin{proof}[Proof of Theorem \ref{main1}]
As before, for a given $x_0 \in \Rn$ with $|x_0| =R$, we note that the rescaled function $w$ defined by
\begin{equation}\label{res5}
w(x,t)= u(Rx +x_0, R^2t)
\end{equation}
solves
 \eqref{er1} in $Q_4$ with $\tilde V(x, t)= R^2 V(Rx + x_0, R^2 t)$. Since we obviously have 
\begin{equation*}
||\tilde V||_{\infty} \leq R^2 ||V||_\infty,
\end{equation*}
which in particular implies  
\[
||\tilde V||_{\infty}^{2/3} \leq R^{4/3} ||V||_\infty^{2/3},
\]
by an application of Proposition \ref{2sphere2} with $r=1/R$, the desired estimate \eqref{dec1} follows by repeating the arguments in  the proof of Theorem \ref{main2} above.
	
\end{proof}
	
We conclude this work with the 
	
\begin{proof}[Proof of Corollary \ref{landis1}]
	 
by rescaling, without loss of generality we assume that the hypothesis of Theorem \ref{main1} holds. Now if   $u(\cdot, 0) \not \equiv 0$ in $B_1$, then  from \eqref{dec1} it follows that  for $|x_0|=R$ and all $R$ large enough, 
\[
\int_{B_1(x_0)} u^2(x, 0) dx >> e^{-R^{4/3+\ve}}.
\]
Since this  contradicts \eqref{dec2}, we conclude that it must be $u(x, 0) \equiv 0$ for $x \in B_1$. By the space-like strong uniqueness result in \cite{EF}, \cite{EFV} it follows that $u(\cdot, 0) \equiv 0$ in $\R^n$. Applying the backward uniqueness result in \cite{Ch}, \cite{ESS}, \cite{Po}, we finally conclude that $u \equiv 0$ in $\Rn \times [0, \infty)$. 	
	
\end{proof}

\end{document}